\documentclass{amsart}
\usepackage{amssymb, amsmath, latexsym}

\renewcommand{\baselinestretch}{\baselinestretch}
\renewcommand{\baselinestretch}{1.1}
\numberwithin{equation}{section}

\newtheorem{thm}{Theorem}[section]
\newtheorem{lem}[thm]{Lemma}
\newtheorem{cor}[thm]{Corollary}
\newtheorem{prop}[thm]{Proposition}

\theoremstyle{definition}
\newtheorem{defn}[thm]{Definition}

\theoremstyle{remark}
\newtheorem{rmk}[thm]{Remark}
\newtheorem{exam}[thm]{Example}

\numberwithin{equation}{section}

\newcommand{\ra}{{\rightarrow}}

\newcommand{\gen}{\text{gen}}

\newcommand{\ord}{\text{ord}}

\newcommand{\z}{{\mathbb Z}}
\newcommand{\q}{{\mathbb Q}}

\newcommand{\scale}{\mathfrak s}
\newcommand{\norm}{\mathfrak n}

\newcommand{\h}{{\mathbb H}}

\newcommand{\plane}{{\mathbb A}}
\newcommand{\Lp}{\Lambda_{p}}
\newcommand{\lp}{\lambda_{p}}
\newcommand{\Lpinv}{\Gamma_p^L}

\newcommand{\w}{\mathfrak w}
\newcommand{\lab}{\text{label}}
\newcommand{\lb}{[\![\,\,}
\newcommand{\rb}{\,\,]\!]}
\newcommand{\bI}{\mathbf I}
\newcommand{\bA}{\mathbf A}
\newcommand{\gH}{\mathfrak H}
\newcommand{\bsim}{\!\!\sim}
\begin{document}

\title{Class numbers of ternary quadratic forms}

\author{Wai Kiu Chan}
\address{Department of Mathematics and Computer Science, Wesleyan University, Middletown CT, 06459, USA}
\email{wkchan@wesleyan.edu}

\author{Byeong-Kweon Oh}
\address{Department of Mathematical Sciences and Research Institute of Mathematics, Seoul National University, Seoul 151-747, Korea}
\email{bkoh@snu.ac.kr}
\thanks{This work of the second author was supported by the National Research Foundation of Korea(NRF) grant funded by the Korea government(MEST) (No. 2011-0016437).}

\subjclass[2010]{Primary 11E12, 11E20, 11E41}

\keywords{Class numbers, quadratic forms}


\begin{abstract}
G.L. Watson \cite{watson1, watson2} introduced a set of transformations, called Watson transformations by most recent authors, in his study of the arithmetic of integral quadratic forms.  These transformations change an integral quadratic form to another integral quadratic form  with a smaller discriminants, but preserve many arithmetic properties at the same time.  In this paper, we study the change of class numbers of positive definite ternary integral quadratic formula along a sequence of Watson transformations, thus providing a new and effective way to compute the class number of positive definite ternary integral quadratic forms.  Explicit class number formulae for many genera of positive definite ternary integral quadratic forms are derived as illustrations of our method.
\end{abstract}

\maketitle

\section{Introduction}

Determining the class number of a positive definite integral quadratic form is a classical and important problem in number theory.  Ternary integral quadratic forms receive much attention because of their many connections to other areas of mathematics.   A notable one among all these connections is the correspondence between ternary integral quadratic forms and orders in quaternion algebras.   In the case of ternary quadratic forms over $\z$, this correspondence leads to a bijection between similarity classes of positive definite ternary quadratic forms over $\z$ and isomorphism classes of orders in definite quaternion algebras over $\q$.   Because of this bijection,  computing the class numbers of positive definite ternary quadratic forms over $\z$ is tantamount to determining the type numbers of orders in definite quaternion algebras over $\q$.  By applying the Selberg Trace Formula, Pizer \cite{p2} obtains explicit formulae for the type numbers of all Eichler orders (they are called  {\em canonical orders} in \cite{p2}). A formula for all orders is obtained by K\"{o}rner \cite{ok}, but for numerical applications his formula requires the computation of the so called restricted embedding numbers of quadratic orders into quaternion orders, which can be achieved only for some special orders using results of \cite{hiji, P, p1, p2}.

In this paper, we look at the problem of computing the class number of positive definite ternary quadratic form from a different perspective.  The backbone of our approach is a set of transformations, now called Watson Transformations, which is first used by Watson in his doctoral thesis \cite{watson1} and is first in print in his paper \cite{watson2}.   The precise definition of these transformations will be given in Section \ref{Watson}.  They have been reformulated in the geometric language of quadratic spaces and lattices by many recent authors (see, for example \cite{ce, co}), and it is this language we will be using to conduct our discussion throughout this paper.   Unexplained notation and terminology from the theory of quadratic spaces and lattices will follow those of O'Meara's book \cite{om}.  For convenience, a quadratic space is always a positive definite quadratic space over the field of rational numbers $\q$, and the term ``lattice" always refers to a $\z$-lattice on a (not necessarily fixed) quadratic space.  For a lattice $L$, $\gen(L)$ will denote the genus of $L$, and $\gen(L)/\bsim$ is the set of (isometry) classes in $\gen(L)$.  The latter is a finite set and its cardinality is called the class number of $L$, denoted $h(L)$.  The class of $L$ in $\gen(L)$  is denoted by $[L]$.  We will refer to the lattice $L$ as ``primitive" if its scale ideal $\scale(L)$ is  $\z$.  The norm ideal of $L$, denoted $\norm(L)$, is the ideal of $\z$ generated by the set $Q(L)$.  We write $L \cong A$ whenever $A$ is a Gram matrix of $L$, and the discriminant $dL$ is defined to be the determinant of $A$.  A diagonal matrix with $a_1, \ldots, a_n$ on the diagonal is denoted by $\langle a_1, \ldots, a_n \rangle$.

Let $L$ be a primitive ternary lattice, and $m$ be a positive integer.     The Watson transformation at $m$ first takes a sublattice $\Lambda_{m}(L)$, and then scales the quadratic form on this sublattice so that the end result is a primitive ternary lattice denoted $\lambda_{m}(L)$.  It follows from the properties for the Watson transformations developed in \cite{ce, co} that  we can always ``descend" $L$, via a sequence of Watson transformations at different primes or at 4, to a primitive ternary maximal lattice $K$ satisfying some specific local conditions.  We call these lattices {\em stable} and they will be discussed thoroughly in Section \ref{stable}.   In this paper, we will address the important question of determining $h(L)$ from the information we could gather from $\gen(K)$.

Let $p$ be an odd prime.  For any $N \in \gen(\Lp(L))$, let $\Lpinv(N)$ be the set of lattices $M \in \gen(L)$ such that $\Lp(M) = N$.  Then $\gen(L)/\bsim$ is the disjoint union of the classes in $\Lpinv(N)$, where $N$ runs through a complete set of class representatives in $\gen(\Lp(L))$.   Theorem \ref{class} provides explicit formulae for the size of each $\Lpinv(N)/\bsim$ in terms of several effectively computable invariants derived from a set of data called the label of $N$ (see Definition \ref{label}), which depends only on the order of the orthogonal group of $N$ and the symmetries of $N$.   As a matter of fact, we obtain much more in Theorem \ref{class}: we have explicit formulae for the number of classes of lattices in $\Lpinv(N)$ whose isometry groups are of a given order.  In Section \ref{changeoflabels}, we will describe how to determine the labels of the lattices in $\Lpinv(N)$, and in Section \ref{stable} we completely determine the labels of all the classes in the genus of a stable ternary lattice.  So, all these together provides an effective solution to the problem of computing the class number of $L$ using the labels of the lattices in $\gen(K)$, if $L$ descends to a stable lattice $K$ via a sequence of Watson transformations at the odd primes.

The remaining task is to obtain the analogs of all the aforementioned results for the Watson transformations at the prime 2.   Although the line of attack in this case will be essentially the same, different tactics will be employed at various steps of the proof due to the lack of uniqueness of Jordan structures at the prime 2.   We will address them in a second paper.

\section{Watson transformations and class numbers} \label{Watson}

Let $m$ be a positive integer.  For any lattice $L$, let
$$\Lambda_m(L) = \{x \in L : Q(x + y) \equiv Q(y) \mbox{ mod $m$ for all } y \in L\},$$
and for every prime number $q$, let
$$\Lambda_m(L_q) = \{x \in L_q : Q(x + y) \equiv Q(y) \mbox{ mod $m$ for all } y \in L_q\}.$$
It is clear that $\Lambda_m(L)$ is a sublattice of $L$ and $\Lambda_m(L) \subseteq \{x \in L : Q(x) \equiv 0 \mod m\}$.  Moreover, $\Lambda_m(L)_q = \Lambda_m(L_q)$ for every prime number $q$.  The readers are referred to \cite{ce} and \cite{co} for more properties of the operators $\Lambda_m$.  We denote by $\lambda_m(L)$ the primitive lattice obtained by scaling the quadratic map on $\Lambda_m(L)$ suitably.  The mappings $\lambda_m$ collectively are called the Watson transformations.\footnote{Note that the Watson transformations $\lambda_m$ defined in \cite{ce} and \cite{co} are slightly different than the ones we use here: the lattices $L$ and  $\lambda_m(L)$ in \cite{ce} and \cite{co} are even primitive, that is, their norm ideals are $2\z$.  But this difference can be easily rectified by scaling the quadratic maps by  suitable 2-powers.}

In what follows, $p$ is always a prime number.  The following describes what $\Lambda_p$ does to $L$ when $p$ is odd.

\begin{lem} \label{lambda}
Let $p$ be an odd prime.  Suppose that $L_p = M_p \perp M'_p$, where $M_p$ is unimodular and $\mathfrak n(M'_p) \subseteq p\z_p$.  Then
$$\Lambda_{p}(L)_p = pM_p \perp M'_p.$$
In particular, if $m$ is an odd squarefree positive integer and $\ord_p(dL) \leq 1$ for all $p \mid m$, then $\lambda_m^2(L) = L$.
\end{lem}
\begin{proof}
The first assertion is essentially \cite[Lemma 2.3]{ce} or \cite[Lemma 2.1]{co}, which has the second assertion as a direct consequence.
\end{proof}

\begin{lem} \label{conf}
Let $L$ be a lattice on a quadratic space $V$.  Then $\sigma \circ \Lambda_{p}(L)=\Lambda_{p} \circ \sigma(L)$ for every $\sigma \in O(V)$.
\end{lem}
\begin{proof} Let $x \in \Lambda_{p}(L)$ and $z \in \sigma(L)$.  Take $w \in L$ such that $\sigma(w)=z$.
Then
$$
Q(\sigma(x)+z)=Q(x+w)\equiv Q(x)=Q(\sigma(x)) \mod{p},
$$
hence $\sigma(x) \in \Lambda_{p}(\sigma(L))$.  The lemma follows immediately from the observation that $\sigma(\Lambda_{p}(L))$ and $\Lambda_{p}(\sigma(L))$ have the same discriminant.
\end{proof}

\begin{cor} \label{con2} For any lattice $L$, the restriction map induces an injective group homomorphism from $O(L)$ into $O(\Lambda_{p}(L))$.
\end{cor}
\begin{proof} This is clear. \end{proof}

Henceforth, $L$ will always be a primitive ternary lattice on a quadratic space $V$.   It is easy to see that every lattice in $\gen(\Lp(L))$ is of the form $\Lp(M)$ for some $M \in \gen(L)$.  Therefore, $\Lp$ induces a surjective function from $\gen(L)$ onto $\gen(\Lp(L))$, and by Lemma \ref{conf} it induces a surjective function from $\gen(L)/\bsim$ to $\gen(\Lp(L))/\!\!\sim$.     Since $\Lp(L)$ and $\lp(L)$ are only different by a scaling on the quadratic maps,  $\lp$ also induces surjective functions from $\gen(L)$ onto $\gen(\lp(L))$ and from $\gen(L)/\bsim$ onto $\gen(\lp(L))/\bsim$, respectively.

Let $\Lpinv(N)$ be the set of lattices $M \in \gen(L)$ such that $\Lp(M) = N$, and $\Lpinv(N)/\bsim$ be the set of classes $[M]$ in $\gen(L)$ such that $\Lp(M) = N$.  Clearly,
$$h(L) = \sum_{[N]\in \gen(\Lp(L))} \vert \Lpinv(N)/\bsim \vert.$$

For simplicity, we let $h_{2d}(N)$ be the number of classes in $\Lpinv(N)$ having an isometry group of order $2d$, but keep in mind that this number depends also on $L_p$.  It is clear that
\begin{equation}
\vert \Lpinv(N)/\bsim\vert =  \sum_{d}{}^{'} h_{2d}(N), \label{h2}
\end{equation}
where in the summation $\sum'$, $d$ runs through all the positive divisor of $\vert O^+(N) \vert$.

Let $\sigma \in O(N)$ and $M \in \Lpinv(N)$.  Since $\Lp(\sigma(M)) = \sigma(\Lp(M)) = \sigma(N) = N$, $\sigma(M)$ belongs to $\Lpinv(N)$.  So, $O(N)$ acts on the set $\Lpinv(N)$.   Moreover, Lemma \ref{conf} implies that if $\tau(M) \in \Lpinv(N)$ for some isometry $\tau$ of $V$, then $\tau$ is in $O(N)$.  Therefore,
\begin{equation}
\sum_{d}{}^{'} h_{2d}(N) = \frac{1}{\vert O^+(N) \vert} \sum_{\sigma \in O^+(N)} \vert \Lpinv(N)_\sigma \vert, \label{h1}
\end{equation}
where $\Lpinv(N)_\sigma$ is the set of fixed points of $\sigma$.  The last equality comes from the observation that $\Lpinv(L)_\sigma = \Lpinv(L)_{-\sigma}$.

\section{The cardinality of $\Lpinv(N)$}

Since the size of the orbit containing $M \in \Lpinv(N)$ under the action of $O(N)$  is $\vert O(N)\vert/\vert O(M)\vert$, we have the equation
\begin{equation}
\sum_d{}^{'} \frac{\vert O^+(N)\vert}{d} h_{2d}(N) = \vert \Lpinv(N) \vert. \label{first}
\end{equation}

\begin{prop} \label{ani1-1}
For any $N \in \gen(\Lambda_p(L))$
$$\vert \Lpinv(N) \vert = \frac {\w(L)}{\w(\Lambda_p(L))},$$
where $\w(L)$ and $\w(\Lp(L))$ are the mass of $\gen(L)$ and $\gen(\Lp(L))$ respectively.
\end{prop}
\begin{proof}
Since $O(N)$ acts on the set $\Lpinv(N)$, we have
$$\vert \Lpinv(N) \vert = \sum_{[M] \in \Lpinv(N)} \frac{\vert O(N)\vert}{\vert O(M)\vert}.$$
However, if $N'$ is another lattice in $\gen(\Lp(L))$, then it is easy to see that there is a bijection between $\Lpinv(N)$ and $\Lpinv(N')$.  Therefore, $\vert \Lpinv(N) \vert$ is independent of the choice of $N$ in $\gen(\Lp(L))$, and hence
\begin{eqnarray*}
\w(\Lp(L)) \, \vert \Lpinv(N) \vert & = & \sum_{[N'] \in \gen(\Lp(L))} \frac{1}{\vert O(N')\vert}\sum_{[M] \in \Lpinv(N)}
                                            \frac{\vert O(N)\vert}{\vert O(M) \vert}\\
    & = & \sum_{[N'] \in \gen(\Lp(L))} \sum_{[M] \in \Lpinv(N')} \frac{1}{\vert O(M) \vert}\\
    & = &  \sum_{[M] \in \gen(L)} \frac{1}{\vert O(M) \vert}\\
    & = & \w(L).
\end{eqnarray*}
\end{proof}

From now on, till the end of the paper, {\em the prime $p$ is always assumed to be odd} and $\ord_p(dL) \geq 2$.  Using the Minkowski-Siegel mass formula \cite[Theorem 6.8.1]{ki}, we have
$$\frac{\w(L)}{\w(\Lambda_p(L))}=\left(\frac{dL}{d(\Lambda_p(L))}\right)^2 \frac{\alpha_p(\Lambda_p(L)_p,\Lambda_p(L)_p)}{\alpha_p(L_p,L_p)},$$
where $\alpha_p(\, , \,)$ are the local densities.   These local densities can be computed by \cite[Theorem 5.6.3]{ki}.  The values of $\frac{\w(L)}{\w(\Lp(L))}$ are displayed in the Table I.  They are arranged by the Jordan decomposition of $L_p$.  The quantity $e_{ij}$ in the table is defined as follows.  Suppose that $L_p \simeq \langle \epsilon_1, p^{\alpha}\epsilon_2, p^{\beta} \epsilon_3\rangle$, where $\alpha \le \beta$ and $\epsilon_i \in \z_p^{\times}$ for all $i$.  Then
$$e_{ij} = \begin{cases}
1 & \mbox{ if $-\epsilon_i\epsilon_j \in (\z_p^{\times})^2$},\\
-1 & \mbox{ otherwise}.
\end{cases}$$

\begin{table} [ht]
\begin{centerline}{\tabcolsep=3pt
\begin{tabular}{|c|c|c||c|c|c|}
\hline \rule[-2mm]{0mm}{8mm}
{}&$\alpha,\beta$&$\frac{\w(L)}{\w(\Lambda_p(L))}$&{}& $\alpha,\beta$&$\frac{\w(L)}{\w(\Lambda_p(L))}$\\
\hline \hline
\rule[-2mm]{0mm}{6mm} {\bf(1)} &$\ \alpha=0, \beta=2$ &$\frac{p(p+e_{12})}2$& {\bf (2)} &$\ \alpha=0, \beta \ge 3$ &$p^2$  \\
\hline
\rule[-2mm]{0mm}{6mm} {\bf (3)} &$\ \alpha=\beta=1$ &$1$ & {\bf (4)} &$\ \alpha=1,\beta=2$ & $\frac{p-e_{13}}2$ \\
\hline
\rule[-2mm]{0mm}{6mm} {\bf (5)}&$\ \alpha=1, \beta \ge 3$ &$p$& {\bf (6)}&$\ \alpha=\beta=2$ &$\frac{p(p+e_{23})}2$  \\
\hline
\rule[-2mm]{0mm}{6mm} {\bf (7)}&$\ \alpha=2, \beta \ge 3$ &$\frac{p(p-e_{12})}2$& {\bf (8)} &$\ \alpha\ge 3$ &$p^2$\\
\hline
\end{tabular}}
\end{centerline}
\vskip 0.2cm
\center{\rm Table I}
\end{table}

\section{Isometry groups} \label{isometry}

We digress in this section to collect some results concerning the isometries of a ternary lattice which are useful for the subsequent discussion.  Throughout this section, $K$ is a primitive ternary lattice.   Given a nonzero vector $x \in K$, the associated symmetry is denoted by $\tau_x$.  We let $S(K)$ be the set of symmetries of $K$.  Since the conjugate of a symmetry of $K$ by any isometry in $O(K)$ is still a symmetry, $S(K)$ is decomposed into finitely many disjoint conjugacy classes under the conjugate action by $O(K)$.  In this section, when we present $S(K)$ explicitly by listing its elements, we will do so by presenting it as the disjoint union of these conjugacy classes.


By a result of Minkowski \cite{Min}, $\vert O(K) \vert$ cannot be larger than 48.    Let $\mathbf I$ be the standard cubic lattice, $\mathbf A$ be the root lattice of Type $A_3$, and $\mathbf J$ be the primitive adjoint of $\bA$; so
$$\bI \cong \langle 1,1,1\rangle, \quad \bA \cong \begin{pmatrix} 2 & 1 & 0\\ 1 & 2 & 1\\ 0 & 1 & 2 \end{pmatrix}  \mbox{ and }
\mathbf J \cong \begin{pmatrix} 3& -1& -1\\ -1 & 3 & -1\\ - 1& -1 & 3 \end{pmatrix}.$$
The isometry groups of all three lattices have order 48--in fact, they are isomorphic--and they are generated by $-I$ and symmetries.  If $\{x_1, x_2, x_3\}$ is the basis which yields any one of the above Gram matrices, then
$$S(\bI) = \{\tau_{x_1}, \tau_{x_2}, \tau_{x_3}\}\cup \{\tau_{x_i \pm x_j} : 1\leq i < j \leq 3\},$$
$$S(\bA) = \{\tau_{x_1}, \tau_{x_2}, \tau_{x_3}, \tau_{x_1 - x_2}, \tau_{x_2-x_3}, \tau_{x_1 - x_2 + x_3}\}\cup \{\tau_{x_1-x_3}, \tau_{x_1+x_3}, \tau_{x_1-2x_2+x_3}\},$$
and
\begin{eqnarray*}
S(\mathbf J) & = & \{\tau_{x_i + x_j} : 1\leq i < j \leq 3\}\cup \\
    & & \{\tau_{x_1 - x_2}, \tau_{x_1 - x_3}, \tau_{x_2 - x_3},  \tau_{x_1 + 2x_2 + x_3}, \tau_{x_1 + x_2 + 2x_3}, \tau_{2x_1 + x_2 + x_3}\}.
\end{eqnarray*}

It is direct to check that $\Lambda_2(\mathbf J) = 2\bI$ and $\Lambda_2(\bI) = \bA$; so $\lambda_2(\mathbf J) = \frac{1}{2}\Lambda_2(\mathbf J)$.  Now, suppose that $L$ is a ternary lattice such that $\Lp(L) = p\mathbf J$.  Let $M = \lambda_2(L)$ and $E = \lambda_2(M)$.  It is not hard to see that $\Lp(M) = p\bI$ and $\Lp(E) = p\bA$, and the $\lambda_2$ transformation induces bijections
\begin{equation}
\Lpinv(p\mathbf J)/\bsim \ \longrightarrow \ \Gamma_p^M(p\bI)/\bsim \ \longrightarrow \ \Gamma_p^E(p\bA)/\bsim. \label{bijection}
\end{equation}
For every $G \in \Gamma_p^E(p\bA)$, define a ternary lattice $G^*$ by setting $G^*_2 = 2G_2^\sharp$, where $\sharp$ denotes the dual, and $G^*_q= G_q$ for all $q \neq 2$.  Then, $E^* = L$, and $*$ induces a bijection from $\Gamma_p^E(p\bA)/\bsim$ back to $\Lpinv(p\mathbf J)/\bsim$ such that $O(G) = O(G^*)$ for all $G \in \Gamma_p^E(p\bA)$.    It then follows from Corollary \ref{con2} that $O(U) = O(\lambda_2(U))$ for any $U$ in either $\Lpinv(p\mathbf J)$ or $\Gamma_p^M(p\bI)$.

Suppose that every element in $O(K)$ has order $\leq 2$.   If $\sigma \neq -I$, then either $\sigma$ or $-\sigma$ is a symmetry.  As a result,  $O(K)$ is an elementary 2-group which is generated by $-I$ and the symmetries.  Let $\tau_u$ and $\tau_v$ be two different symmetries in $O(K)$.  Since $\tau_u$ and $\tau_v$ commutes, $u$ and $v$ must be orthogonal.  This shows that $\vert O(K)\vert$ is at most 8.  Particularly, if $\vert O(K) \vert = 8$, $O(K)$ contains exactly three symmetries $\tau_w$, $\tau_u$, $\tau_v$, and $w, u, v$ are mutually orthogonal in $K$.  In particular, $O(K)$ is isomorphic to the abelian 2-group $\z_2\oplus \z_2\oplus \z_2$.

Now, suppose that $O(K)$ has an isometry $\sigma$ of order 3.   As a $\z[\sigma]$-module, $K$ is isomorphic to either $(\z[\zeta_3], 1)$ or $\z[\zeta_3]\oplus \z$, where $\zeta_3$ is a primitive third root of unity (see \cite{cr}).  Accordingly, $K$ has a basis $\{x_1, x_2, x_3\}$ such that
\begin{equation*}
\sigma(x_1)=x_2, \quad \sigma(x_2)=-x_1-x_2 \quad \text{and} \quad \sigma(x_3)=x_1+x_3
\end{equation*}
or
\begin{equation*}
\sigma(x_1)=x_2, \quad \sigma(x_2)=-x_1-x_2 \quad \text{and} \quad \sigma(x_3)=x_3,
\end{equation*}
and the associated symmetric matrix $(B(x_i, x_j))$ is
$$ K_1(a,b):=\begin{pmatrix} 2a&-a&-a\\-a&2a&0\\-a&0&b\end{pmatrix} \quad \text{or} \quad K_2(a,b):= \begin{pmatrix} 2a&-a&0\\-a&2a&0\\0&0&b\end{pmatrix}$$
for a pair of  relatively prime positive integers $a$ and $b$.  In the three special cases when $K_1(1,1) \cong \bI$, $K_1(1, 2) \cong \bA$, and $K_1(4,3) \cong \mathbf J$, the isometry groups have order 48.

\begin{lem} \label{1224}
Let $a, b$ be relatively prime positive integers.  Then
\begin{enumerate}
\item[(a)] $\vert O(K_2(a,b)) \vert = 24$;

\item[(b)] $\vert O(K_1(a,b))\vert = 12$ unless $(a,b) = (1,1), (1,2)$, or $(4,3)$.
\end{enumerate}
\end{lem}
\begin{proof}
Part (a) is clear, since $K_2(a,b)$ is the orthogonal sum of $\z x_1 + \z x_2$ and $\z x_3$.

For part (b), note that $b > 2a/3$ because $K_1(a,b)$ is positive definite.  Let $G$ be the subgroup of $O(K_1(a,b))$ that is generated by $\tau_{x_1}$, $\tau_{x_2}$, $\tau_{x_1+x_2}$, and $-I$.  Our goal is to show that $O(K_1(a,b))$ is equal to this subgroup $G$, which has order 12, unless $(a,b)$ is one of the three exceptional cases.

We first handle the case when $b > 2a$. In this case, $\{x_1, x_2, x_3\}$ is a Minkowski reduced basis, and therefore the minimal vectors in $K_1(a,b)$ are $\pm x_1, \pm x_2$, and $\pm (x_1 + x_2)$.  Let $\sigma$ be an isometry of $K_1(a,b)$.  Since $\sigma$ must permute the minimal vectors, $\sigma$ induces an isometry on the sublattice $\z x_1 + \z x_2$, and hence we may assume that $\sigma(x_i) = x_i$ for $i = 1, 2$.  Since $z: = 2x_1 + x_2 + 3x_3$ spans the orthogonal complement of $\z x_1 + \z x_2$, therefore $\sigma(z) = \pm z$.  A direct computation shows that $\sigma(z) = -z$ is impossible, thus $\sigma(z) = z$ and so $\sigma = I \in G$.

Now, let us assume that $b < 2a$.  The Gram matrix of $K_1(a,b)$ with respect to the new basis $y_1 = x_1 + x_3$, $y_2 = x_1 + x_2 + x_3$, $y_3 = x_3$ is
$$\begin{pmatrix} b & b - a & b-a\\ b-a & b & b-a\\ b-a & b-a & b \end{pmatrix}.$$
Suppose that $u: = \alpha y_1 + \beta y_2 + \gamma y_3$ is a primitive vector of $K_1(a,b)$ with $Q(u) = b$, that is
$$b = b(\alpha^2 + \beta^2 + \gamma^2) + 2(b-a)(\alpha\beta + \beta\gamma + \gamma\alpha).$$
If $\alpha^2 + \beta^2 + \gamma^2 = \vert \alpha\beta + \beta\gamma + \gamma\alpha\vert$, then $\alpha = \beta = \gamma \in \{0, 1, -1\}$.  As a result, $4b = 3a$ and so $(a,b)= (4,3)$ which is a contradiction.  Thus, we may assume that
$$\alpha^2 + \beta^2 + \gamma^2 \geq \vert \alpha\beta + \beta\gamma + \gamma\alpha\vert + 1.$$
Then
\begin{eqnarray*}
b & = & b + b(\alpha^2 + \beta^2 + \gamma^2 -1) + 2(b-a)(\alpha\beta + \beta\gamma + \gamma\alpha)\\
    & = & \begin{cases}
            b + (2a - b)(\alpha^2 + \beta^2 + \gamma^2 - 1) & \mbox{ if $b > a$};\\
            b + (3b - 2a)(\alpha^2 + \beta^2 + \gamma^2 - 1) & \mbox{ if $b < a$}.
          \end{cases}
\end{eqnarray*}
This shows that if $(a,b)$ is not one of the three exceptional pairs, then $\pm y_1, \pm y_2, \pm y_3$ are all the vectors $u \in K_1(a,b)$ such that $Q(u) = b$.  It is direct to check that both $G$ and $O(K_1(a,b))$ act on these six vectors, and $G$ permutes them transitively.  Therefore, for any $\sigma \in O(K_1(a,b))$, there exists $\tau \in G$ such that $\sigma\tau = I$.  Thus $\sigma \in G$ as desired.
\end{proof}

Excluding the three special cases $K_1(1,1), K_1(1,2)$ and $K_1(4,3)$,  $O(K_1(a,b))$ and $O(K_2(a,b))$ are generated by $-I$ and symmetries, and we have
$$S(K_1(a,b))=\{\tau_{x_1}, \tau_{x_2}, \tau_{x_1+x_2}\},$$
and
$$S(K_2(a,b))=\{\tau_{x_1}, \tau_{x_2}, \tau_{x_1+x_2}\}\cup  \{\tau_{x_1+2x_2}, \tau_{2x_1+x_2}, \tau_{x_1-x_2}\}\cup \{  \tau_{x_3} \}.$$
Note that $\tau_{x_3}$ is the only symmetry in the center of $O(K_2(a,b))$.  Also, $O(K_1(a,b)) \cong \z_2\oplus D_3$ and $O(K_2(a,b)) \cong \z_2\oplus \z_2\oplus D_3$.  As a result, both groups do not have any element of order 4.

The following is an immediate consequence of the proof of Lemma \ref{1224}.

\begin{cor}
Let $K$ be a primitive ternary lattice whose isometry group has order $48$.  Then $K \cong \bI, \bA$, or $\mathbf J$.
\end{cor}

Now suppose that the order of the isometry $\sigma$ is $4$. Then by \cite[Proposition 4]{t}, there is a basis $\{z_1,z_2,z_3\}$ of $K$ with respect to which $\sigma$ is represented by one of the following four matrices:
$$\begin{pmatrix} 1 & 0 & 0 \\ 0 & 0 & -1\\ 0 & 1 & 0 \end{pmatrix},\ \begin{pmatrix} -1 & 0 & 0 \\ 0 & 0 & -1\\ 0 & 1 & 0 \end{pmatrix},\
\begin{pmatrix} 1 & 0 & 1 \\ 0 & 0 & -1\\ 0 & 1 & 0 \end{pmatrix}, \ \begin{pmatrix} -1 & 0 & -1 \\ 0 & 0 & 1\\ 0 & -1 & 0 \end{pmatrix}.$$
For the first two matrices, it is easy to see that there is a basis $\{x_1, x_2, x_3\}$ of $K$ such that
$$\sigma(x_1)=-x_2, \ \sigma(x_2)=x_1, \ \sigma(x_3)=\pm x_3,$$
which means that $K$ is isometric to
$$K_3(a,b): =  \langle a,a,b\rangle.$$
For the third matrix, let $x_1 = z_1 - 2z_2$, $x_2 = -z_1 + 2z_3$, and $x_3 = z_2 - z_3$.  Then
$$\sigma(x_1)=-x_2, \ \sigma(x_2)=x_1, \ \sigma(x_3)=x_2+x_3,$$
implying that $K$ is isometric to
$$K_4(a,b):= \begin{pmatrix}2a&0&-a\\0&2a&-a\\-a&-a&b\end{pmatrix}.$$
For the fourth, we take $x_1 = z_1$, $x_2 = -z_3$, and $x_3 = z_1 - z_2$ so that
$$\sigma(x_1)=x_2, \ \sigma(x_2)=x_3, \ \sigma(x_3)=-x_1-x_2-x_3;$$
thus $K$ is isometric to
$$\begin{pmatrix} a&b&-a-2b\\b&a&b\\-a-2b&b&a\end{pmatrix} \cong  \begin{pmatrix} 2a+2b&0&-a-b\\0&2a+2b&-a-b\\-a-b&-a-b&a\end{pmatrix}=K_4(a+b,a).$$
Note that $K_3(1,1) \cong \bI$, $K_4(1,2) \cong \bA$, and $K_4(2,3) \cong \mathbf J$.  It is not hard to see that these are the only cases for which $\vert O(K_3(a,b))\vert$ and $\vert O(K_4(a,b))\vert$ are equal to 48.

\begin{lem}
The isometry groups of $K_3(a,b)$ and $K_4(a,b)$ have order $16$, except for $K_3(1,1)$, $K_4(1,2)$, and $K_4(2,3)$.
\end{lem}
\begin{proof}
Suppose that $K$ is either $K_3(a,b)$ or $K_4(a,b)$, but not one of the three exceptional lattices.  Then $O(K)$ contains at least five symmetries (see below); thus $\vert O(K)\vert > 8$.  Then, by \cite{t}, $\vert O(K) \vert = 12, 16$ or 24. Suppose that $\vert O(K) \vert$ is either 12 or 24.  This means that $K$ must be also of the form $K_i(c,d)$, for $i = 1$ or 2.  But then, as indicated earlier, $O(K)$ would not have any element of order 4, which is a contradiction.  Therefore, $\vert O(K) \vert = 16$ as claimed.
\end{proof}

Excluding the special cases $O(K_3(1,1))$, $O(K_4(1,2))$, and $O(K_4(2,3))$, both $O(K_3(a,b))$ and $O(K_4(a,b))$ are also generated by $-I$ and the symmetries, and
$$S(K_3(a,b))=\{\tau_{x_1}, \tau_{x_2}\}\cup \{ \tau_{x_1+x_2}, \tau_{x_1-x_2}\}\cup \{ \tau_{x_3}\},$$
and
$$S(K_4(a,b))=\{\tau_{x_1}, \tau_{x_2}\} \cup \{ \tau_{x_1+x_2}, \tau_{x_1-x_2}\}\cup \{ \tau_{x_1+x_2+2x_3}\}.$$
In either case, the center of the orthogonal group contains one and only one symmetry, namely $\tau_{x_3}$ for $K_3(a,b)$ and $\tau_{x_1 + x_2 + 2x_3}$ for $K_4(a,b)$.  Moreover, both $O(K_3(a,b))$ and $O(K_4(a,b))$ are isomorphic to $\z_2\oplus D_4$.

\begin{defn}
An orthogonal system of a ternary lattice $K$ is a set of three mutually commuting symmetries in $O(K)$.
\end{defn}

If $\{\sigma_1, \sigma_2, \sigma_3\}$ is an orthogonal system and $\sigma_i = \tau_{z_i}$ for all $i$, then $z_1, z_2, z_3$ are mutually orthogonal vectors.

\begin{prop} \label{sym}
Suppose that $\vert O(K) \vert$ is divisible by $8$.   Every symmetry of $K$ belongs to an orthogonal system of $K$.
\end{prop}
\begin{proof}
This is done by checking the set of symmetries $S(K)$ for all possible cases.
\end{proof}

We can say more about orthogonal systems when $\vert O(K) \vert = 16$ or 24.  In these two cases, there is a unique symmetry $\tau$ that is in the center of $O(K)$.  If $\sigma$ is another symmetry of $K$ which is not $\tau$, then $\sigma$ belongs to one and only one orthogonal system, and this orthogonal system contains another symmetry $\sigma'$, uniquely determined by $\sigma$ of course, and $\tau$. All of these can be proved by examining the set of symmetries $S(K)$ and writing down all the orthogonal systems of $K$.  For $K_2(a,b)$, the orthogonal systems are
$$\{\tau_{x_1}, \tau_{x_1 + 2x_2},\tau_{x_3}\}, \quad \{\tau_{x_2}, \tau_{2x_1 + x_2},  \tau_{x_3}\}, \quad \{\tau_{x_1+x_2}, \tau_{x_1 - x_2}, \tau_{x_3}\}.$$
For $K_3(a,b)$ and $K_4(a,b)$, their orthogonal systems are
$$\{\tau_{x_1}, \tau_{x_2}, \tau_{x_3}\}, \quad \{\tau_{x_1 + x_2}, \tau_{x_1 - x_2}, \tau_{x_3}\}$$
and
$$\{\tau_{x_1}, \tau_{x_2}, \tau_{x_1 + x_2 + 2x_3}\}, \quad \{\tau_{x_1 + x_2}, \tau_{x_1 - x_2}, \tau_{x_1 + x_2 + 2x_3}\}$$
respectively.

\section{Formulae for $\vert \Lpinv(N)_\sigma \vert$}

In this section, we always assume that $L$ is a primitive ternary lattice and $N = \Lp(L)$.  We will obtain information regarding the cardinality of the set of fixed points $\Lpinv(N)_\sigma$ for each nontrivial isometry $\sigma$ of $N$.  In the case when $\sigma$ is a symmetry we will obtain explicit formulae to compute $\vert \Lpinv(N)_\sigma \vert$.

\subsection{Fixed points of isometries}

Let $K$ be a ternary lattice with $\vert O(K) \vert = 12$ or 24; so $K$ is $K_1(a,b)$ or $K_2(a,b)$ according to Section \ref{isometry}.  Let $\sigma$ be an isometry of $K$ of order 3.  By Section \ref{isometry}, $K$ has a basis $\{x_1, x_2, x_3\}$ such that $\sigma(x_1) = x_2$, $\sigma(x_2) = -x_1 - x_2$, and $\sigma(x_3) = x_1 + x_3$ or $x_3$.  The characteristic polynomial of $\sigma$ is always $x^3 - 1$, which implies that the fixed points of $\sigma$ in $K$, denoted by $K_\sigma$, is a rank 1 sublattice.  Indeed, a straightforward calculation shows that $K_\sigma = \z w$, where
\begin{equation} \label{w3}
w = \begin{cases}
2x_1 + x_2 + 3x_3 & \mbox{ if $K = K_1(a,b)$},\\
x_3 & \mbox{ if $K = K_2(a,b)$},
\end{cases}
\end{equation}
and $Q(w) = 3(3b - 2a)$ or $b$ accordingly.

The primitive sublattice $\z x_1 + \z x_2 \cong \left (\begin{smallmatrix} 2a & -a\\ -a & 2a \end{smallmatrix}\right)$ of $K$ is orthogonal to $K_\sigma$, and so it is in fact the orthogonal complement of $K_\sigma$ in $K$.  Thus, $[K : \z w \perp (\z x_1 + \z x_2)] = 3$ if $\vert O(K) \vert = 12$, whereas $K = \z w \perp (\z x_1 + \z x_2)$ if $\vert O(K) \vert = 24$.

\begin{prop} \label{order3}
Suppose that $\vert O(N) \vert$ is divisible by $3$, and that $p \neq 3$ if $\vert O(N) \vert = 24$.
\begin{enumerate}
\item[(a)]  If $\vert O(N) \vert < 48$, then there is at most one lattice $M \in \Lpinv(N)$ such that $\vert O(M) \vert$ is divisible by $3$.  Moreover, $O(M) = O(N)$ in this case.

\item[(b)] If $\vert O(N) \vert = 48$, then there is at most one class of lattices $M$ in $\Lpinv(N)$ such that $\vert O(M)\vert$ is divisible by $3$.  If, in addition,  $N = p\bI$, then every one of these lattices is isometric to $K_1(1, \frac{p^2 + 2}{3})$ or $K_1(p^2, \frac{2p^2 + 1}{3})$, when $dM$ is $p^2$ or $p^4$ accordingly.
\end{enumerate}
\end{prop}
\begin{proof}
We first handle the case when $\vert O(N) \vert$ is either 12 or 24.  Suppose that $M \in \Lpinv(N)$ has an isometry $\sigma$ of order 3.  If $\vert O(N) \vert = 12$, then of course $O(M) = O(N)$.  Let us assume that $\vert O(N) \vert = 24$ but $\vert O(M) \vert = 12$.    Since $p \neq 3$ under this assumption, we have $N_3 = M_3$ and hence $N_\sigma$ is not an orthogonal summand of $N$, a contradiction.  Thus $O(M) = O(N)$ whenever 3 divides $\vert O(M) \vert$.

Let $M$ and $M'$ be lattices in $\Lpinv(N)$ such that $O(M) = O(M') = O(N)$.  Let $\sigma$ be an isometry of order 3 in $O(N)$.  Suppose that $\vert O(N) \vert = 24$.  Then $M = K_2(a,b)$ and $M' = K_2(c,d)$ for some integers $a, b, c, d$.  Note that, since $p \neq 3$ when $\vert O(N) \vert = 24$,  $p\mid a$ if and only if $p\mid c$ by considering the local structures of $M_p$ and $M'_p$.  Therefore, if $p$ divides $a$, then
$$\langle p^2d \rangle \cong \Lp(M')_\sigma = N_\sigma = \Lp(M)_\sigma \cong \langle p^2 b \rangle.$$
Thus, $b = d$ and, since $dM = dM'$, we have $a = c$ as well.  This shows that $M$ is isometric to $M'$.  But every isometry from $M$ to $M'$ must lie in $O(N)$; therefore $M = M'$.  The argument for the case $p \nmid a$ is similar, except that we have $\Lp(M)_\sigma \cong \langle b \rangle$ in that case.

If $\vert O(N) \vert = 12$, then $M= K_1(e, f)$ and $M' = K_1(g, h)$ for some integers $e, f, g, h$.  Since $dM = dM'$, therefore $e^2(3f - 2e) = g^2(3h - 2g)$.  When $p \neq 3$, we may argue as before to show that $M = M'$.  When $p = 3$, $3\mid e$ (and hence $3\mid g$ as well) since $\ord_3(dM) \geq 2$.  Then we may use (\ref{w3}) to conclude that $M_\sigma$ is in $\Lp(M)$.  Therefore,
$$\langle 3(3e - 2f) \rangle \cong \Lp(M)_\sigma = N_\sigma = \Lp(M')_\sigma \cong \langle 3(3h - 2g) \rangle,$$
which implies $e = g$ and $f = h$.  As is argued before, we have $M = M'$ as consequence.

We now assume that $\vert O(N) \vert = 48$.  By the bijections in (\ref{bijection}), it suffices to deal with the case when $N = p\bI$.  So, suppose that $\Lp(M) = p\bI$ and that $M$ has an isometry of order 3.  Thus $M$ itself cannot be similar to $\bI$ or $\bA$.  Therefore, $M$ is either $K_1(a,b)$ or $K_2(c,d)$ for some suitable integers $a, b, c, d$.  Note that $M_p$ is in Case {\bf (1)} or Case {\bf (6)} described in Table I.  Suppose that $M = K_2(c,d)$.  Since $dM = 3c^2d$ in this case and $\gcd(c,d) = 1$, we have $p = 3$ and $M$ is $K_2(1,27)$ or $K_2(1, 3)$.  Both possibilities lead to a contradiction since neither $\langle 2, 27 \rangle$ nor $\langle 2, 3 \rangle$ is represented by $M_3$.  If $M = K_1(a,b)$, then a similar analysis on discriminant and local representations--over $\z_p$ this time--shows that $M$ is either $K_1(1, \frac{p^2 + 2}{3})$ or $K_1(p^2, \frac{2p^2 + 1}{3})$, depending on whether $dL$ is $p^2$ or $p^4$.
\end{proof}

Let $G$ be a ternary lattice with an isometry $\sigma$ of order 4.   By replacing $\sigma$ by $-\sigma$ if necessary, we can always assume that the characteristic polynomial of $\sigma$ is $(x^2 + 1)(x - 1)$ and $G_\sigma$ is a rank 1 sublattice.  Suppose that $\vert O(N) \vert = 16$.  Then $G = K_3(a,b)$ or $K_4(a,b)$ for some integers $a,b$, and $G$ has a basis $\{x_1, x_2, x_3\}$ such that $\sigma(x_1) = -x_2$, $\sigma(x_2) = x_1$, and $\sigma(x_3) = x_3$ or $x_2 + x_3$.  So,
$$G_\sigma = \begin{cases}
\z x_3 \cong \langle b \rangle & \mbox{ if $G \cong K_3(a,b)$},\\
\z(x_1+ x_2 + 2x_3) \cong \langle 4(b-a) \rangle & \mbox{ if $G \cong K_4(a,b)$}.
\end{cases}$$
In any case, $\z x_1 + \z x_2 \cong \langle 2a, 2a \rangle$ is the orthogonal complement of $G_\sigma$.  As a result, $G = G_\sigma \perp (\z x_1 + \z x_2)$ if $G = K_3(a,b)$, whereas $[G : G_\sigma \perp (\z x_1 + \z x_2)] = 2$ if $G = K_4(a,b)$.

\begin{prop} \label{order4}
Suppose that $N$ has an isometry of order $4$.
\begin{enumerate}
\item[(a)] If $\vert O(N) \vert < 48$, then there is at most one lattice $M \in \Lpinv(N)$ such that $M$ has an isometry of order $4$.  Furthermore, if $\vert O(N) \vert = 16$, then $M = K_3(a,b)$ for some $a, b \in \z$ if and only if $N = K_3(c,d)$ for some $c, d \in \z$.

\item[(b)] If $\vert O(N) \vert = 48$, then there is at most one class of lattices $M \in \Lpinv(N)$ such that each of these $M$ has an isometry of order $4$.  If, in addition, $N  = p\bI$, then all these lattices $M$ are isometric to $K_3(1,p^2)$ if $dM = p^2$, or $K_3(p^2,1)$ if $dM = p^4$.
\end{enumerate}
\end{prop}
\begin{proof}
Suppose that $\vert O(N) \vert = 16$, and that $M$ is a lattice in $\Lpinv(N)$ with an isometry $\sigma$ of order 4 such that $M_\sigma \neq 0$.  If $M = K_3(a,b)$, then of course $N = K_3(c,d)$ for some $c, d \in \z$.  If $M = K_4(a,b)$, then $M_2 = \Lp(M)_2 = N_2$ and $(N_\sigma)_2 = (N_2)_\sigma = (M_2)_\sigma = (M_\sigma)_2$. So, $N_\sigma$ is not an orthogonal summand of $N$, which means that $N$ is $K_4(c,d)$ for some $c, d \in \z$.  The rest of the proof is the same as the one for Proposition \ref{order3}, and we leave it for the readers.

Now, suppose that $\vert O(N) \vert = 48$.  As in Proposition \ref{order3}, we may assume that $N = p\bI$.  Suppose that $M \in \Lpinv(N)$ is a lattice which has an isometry $\sigma$ of order 4.  As a result, $M$ is either $K_3(a,b)$ or $K_4(a,b)$.  We can rule out $K_4(a,b)$ by discriminant consideration.  So, $M$ must be either $K_3(1, p^2) \cong \langle 1, 1, p^2 \rangle$ when $dM = p^2$ or $K_3(p^2,1) = \langle p^2, p^2, 1\rangle$ when $dM = p^4$.
\end{proof}

The following corollary is a direct consequence of Proposition \ref{order3} and Proposition \ref{order4}.

\begin{cor} \label{34fin}
Suppose that $p \neq 3$ when $\vert O(N) \vert = 24$.  If $\vert O(N) \vert < 48$, then  $\vert \Lpinv(N)_{\sigma}\vert \le 1$ for any $\sigma \in O(N)$ of order at least $3$.
\end{cor}

\begin{rmk}
When $\vert O(N) \vert = 48$, Propositions \ref{order3} and \ref{order4} show that
\begin{equation}\label{48}
h_{12}(N) = (1 - \delta_{3p})\left (\frac{1 + \left(\frac{3d_0}{p} \right)}{2}\right) \quad h_{16}(N) = \frac{1 + \left(\frac{d_0}{p} \right)}{2},
\end{equation}
where $\delta_{ij}$ is the Kronecker's delta, $d_0$ is the discriminant of a unimodular Jordan component of $L_p$, and $\left(\frac{}{p}\right)$ is the Legendre symbol.
\end{rmk}

\subsection{Special symmetries}

\begin{lem} \label{trivial}
Let $M$ be an $R$-lattice, where $R$ is either $\z$ or $\z_p$.  Suppose that $\tau_w$ is a symmetry in $O(M)$ with $w$ a primitive vector in $M$. If $Q(w)$ is odd or $R=\z_p$, then $Rw$ is an orthogonal summand of $M$.
\end{lem}
\begin{proof}
Since $w$ is primitive in $M$, there exist vectors $x_2, \ldots, x_n$ in $M$ such that $M=R w+R x_2 +\cdots+R x_n$.  Since $\tau_w \in O(M)$,  $\tau_w(x_i) \in M$ for each $i = 2, \ldots, n$, and it follows that there exists $a_i \in R$ such that $B(w,a_iw+x_i)=0$.   Then $M = Rw \perp M'$, where $M'=R(a_1w+x_2)+\cdots+R(a_1w+x_n)$.
\end{proof}

Let $M$ be a ternary lattice.  Every symmetry $\sigma$ in $O(M)$ is of the form $\tau_x$, where $x$ is a primitive vector of $M$.  We define $Q_M(\sigma)$ to be $Q(x)$.  Note that $Q_M(\sigma)$ is well-defined.  The next technical definition, which depends on the prime $p$ we fix at the outset, is tailored for later discussion.

\begin{defn} \label{length}
Let $\sigma$ be a symmetry of a ternary lattice $M$.  We call $\sigma$ special to $M$ if
$$Q_M(\sigma) \begin{cases}
\not \equiv 0 \mod p& \mbox{ if the unimodular component of $M_p$ has rank 1};\\
\equiv 0 \mod p & \mbox{ otherwise}.
\end{cases}$$
\end{defn}

Note that in the definition if $\sigma = \tau_x$ is special to $M$ with $x$ primitive in $M$, then $\z_px$ must be either the leading or the last component of a Jordan decomposition of $M_p$.

\begin{prop}\label{specialone}
Suppose that $\Lp(L) = N$ and $\sigma$ is a symmetry of $N$.  Then there exists at  most one lattice in $\Lpinv(N)_\sigma$ to which $\sigma$ is special.
\end{prop}
\begin{proof}
Suppose that $\sigma$ is special to a lattice $M \in \Lpinv(N)$.  Choose a primitive vector $w$ of $M$ for which $\sigma = \tau_w$, and let $G$ be the orthogonal complement of $w$ in $M$.  Then $M_p = \z_p w \perp G_p$ and
$$N_p = \Lp(M_p) = \begin{cases}
\z_p pw \perp G_p & \mbox{ if $Q(w) \not \equiv 0 \mod p$},\\
\z_p w \perp pG_p & \mbox{ otherwise}.
\end{cases}$$
Let $M'$ be another lattice in $\Lpinv(N)$ to which $\sigma$ is special.  Then $M'_p = \z_p w' \perp G_p'$, where $\sigma = \tau_{w'}$ with $w'$ primitive in $M'$.   Since $M_p \cong M'_p$,  it follows from the definition of special symmetry that $\z_p w' = \z_p w$.   Moreover, $N_p$ is also equal to $\Lp(M'_p)$, which implies that $G'_p = G_p$, whence $M' = M$.
\end{proof}

\begin{prop} \label{order8}
Suppose that $\Lp(L)=N$ and $\vert O(N)\vert$ is divisible by $8$. Let $M$ be a lattice in $\Lpinv(N)$.
Then $\vert O(M)\vert$ is divisible by $8$ if and only if there is a symmetry in $O(M)$ which is special to $M$.
\end{prop}
\begin{proof}
Suppose that $O(M)$ has a symmetry $\sigma$ which is special to $M$.  Let $w$ be a primitive vector in $M$ such that $\sigma = \tau_w$, and let $G$ be the orthogonal complement of $w$ in $M$.  Then $M_p = \z_p w \perp G_p$, and
$$N_p = \Lp(L_p) = \begin{cases}
\z_p pw \perp G_p & \mbox{ if $Q(w) \not \equiv 0$ mod $p$},\\
\z_pw \perp pG_p & \mbox{ otherwise}.
\end{cases}$$
Since $\vert O(N)\vert$ is divisible by 8, Proposition \ref{sym} shows that there exist symmetries $\tau_u, \tau_v \in O(N)$ where $w, u, v$ are mutually orthogonal.  Then $\tau_u$ and $\tau_v$ must be isometries of $G_p$, which means that both of them are isometries of $M_p$.  Since $N_q = L_q$ for any prime $q$ not equal to $p$, therefore both $\tau_u$ and $\tau_v$ are isometries of $M$.  The subgroup of $O(M)$ generated by $\sigma_w, \sigma_u, \sigma_v$ is of order 8.

Conversely, suppose that $\vert O(M)\vert$ is divisible by 8.  It follows from Lemma \ref{sym} that there are three mutually orthogonal primitive vectors $x, y, z$ in $M$ such that $\tau_x$, $\tau_y$, and $\tau_z$ are in $O(M)$.  Since $M_p = \z_px \perp \z_py \perp \z_p z$ by Lemma \ref{trivial}, exactly one of $\tau_x, \tau_y$, and $\tau_z$ is special to $M$.
\end{proof}

\begin{prop} \label{order12}
Let $M \in \Lpinv(N)$ and suppose that $\vert O(M) \vert = 12$.  Then there is no symmetry in $O(M)$ that is special to $M$.
\end{prop}
\begin{proof}
From Section \ref{isometry}, $M = K_1(a,b)$ with relatively prime positive integers $a$ and $b$.  The discriminant of $M$ is $a^2(3b - 2a)$, and $Q_M(\sigma) = 2a$ for each symmetry $\sigma \in S(M)$.  If $p \neq 3$, then $L_p \cong \langle 2a, 6a, 3(3b - 2a) \rangle$ and so none of the symmetries is special to $M$.

Suppose that $p = 3$.  Since $\Lambda_3(M) = N$, 3 divides the discriminant of $M$.  Hence 3 divides $a$, and so 3 does not divide $b$.  This implies $L_3 \cong \langle b, 2a, 2ab(3b - 2a) \rangle$, and thus none of the symmetries is special to $M$.
\end{proof}

\subsection{Fixed points of a symmetry}

Let $L_p = \langle \epsilon_1, p^\alpha\epsilon_2, p^\beta \epsilon_3\rangle$ as in Table I.  Recall that
$$e_{ij} = \begin{cases}
1 & \mbox{ if $-\epsilon_i\epsilon_j \in (\z_p^{\times})^2$},\\
-1 & \mbox{ otherwise}.
\end{cases}$$
In addition, we define
$$\eta_{ij} = \frac{1 + \left(\frac{\epsilon_i\epsilon_j}{p}\right)}{2}, \mbox{ and } \eta'_{ij} = \frac{1 - \left(\frac{\epsilon_i\epsilon_j}{p}\right)}{2}.$$
In Table II below, the cases {\bf (1)} to {\bf (8)} are divided as in Table I, $Q_N(\sigma)$ is defined in Definition \ref{length}, and $\Delta$ is a nonsquare unit in $\z_p$.  For a pair of $p$-adic integers $a$ and $b$, we write $a \sim b$ if $ab^{-1} \in (\z_p^\times)^2$. A boldface {\bf 1} indicates that there is one lattice in $\Lpinv(L)_\sigma$ to which $\sigma$ is special, under the specified conditions on $Q_N(\sigma)$.

\begin{prop} \label{symmetryfixedpoints}
Let $\sigma$ be a symmetry in $O(N)$.  Then the values of $\vert \Lpinv(N)_\sigma \vert$ are given in Table II.
\begin{table} [ht]
\centering
{\tabcolsep=3pt
\begin{tabular}{|c|c|c|c|c|}
\hline \rule[-2mm]{0mm}{8mm}
 &\textnormal{condition}   & $Q_N(\sigma) \sim p^2\epsilon_3$ & $Q_N(\sigma) \sim p^2\Delta\epsilon_3$ & \\ \cline{2-4}
\raisebox{2.5ex}[0cm][0cm]{\bf(1)} & $\vert \Lpinv(N)_\sigma\vert$ & $\displaystyle{\frac{p - e_{12}}{2}}+ {\mathbf 1}$ &
        $\displaystyle{\frac{p + e_{12}}{2}}$ & \\
\hline \rule[-2mm]{0mm}{8mm}
  &\textnormal{condition}   & $\ord_p(Q_N(\sigma)) \geq 3$ & $\ord_p(Q_N(\sigma)) = 2$ & \\ \cline{2-4}
\raisebox{2.5ex}[0cm][0cm]{\bf(2)} & $\vert \Lpinv(N)_\sigma\vert$ & {\bf 1} & $p$ & \\
\hline \rule[-2mm]{0mm}{8mm}
  &\textnormal{condition}  & $Q_N(\sigma) \sim p^2\epsilon_1$ & $\ord_p(Q_N(\sigma)) = 1$ & \\ \cline{2-4}
\raisebox{2.5ex}[0cm][0cm]{\bf(3)}& $\vert \Lpinv(N)_\sigma\vert$  & {\bf 1} & $1$ & \\
\hline \rule[-2mm]{0mm}{8mm}
   &\textnormal{condition}   & $Q_N(\sigma) \sim p^2\Delta\epsilon_1$ & $Q_N(\sigma) \sim p^2\epsilon_1$ & $\ord_p(Q_N(\sigma)) = 1$ \\ \cline{2-5}
\raisebox{2.5ex}[0cm][0cm]{\bf(4)}& $\vert \Lpinv(N)_\sigma\vert$ & $\eta'_{13}$ & $\eta_{13} + {\mathbf 1}$ &
    $\displaystyle{\frac{p - e_{13}}{2}}$\\
\hline \rule[-2mm]{0mm}{8mm}
  &\textnormal{condition}   & $\ord_p(Q_N(\sigma)) \geq 3$ & $\ord_p(Q_N(\sigma)) = 2$ & $\ord_p(Q_N(\sigma)) = 1$ \\ \cline{2-5}
\raisebox{2.5ex}[0cm][0cm]{\bf(5)}& $\vert \Lpinv(N)_\sigma\vert$  & $1$ & {\bf 1} & $p$ \\
\hline \rule[-2mm]{0mm}{8mm}
   &\textnormal{condition}   & $Q_N(\sigma) \sim p^2\epsilon_1$ & $Q_N(\sigma) \sim p^2\Delta\epsilon_1$ & \\ \cline{2-4}
\raisebox{2.5ex}[0cm][0cm]{\bf(6)}& $\vert \Lpinv(N)_\sigma\vert$ & $\displaystyle{\frac{p - e_{23}}{2}} + {\mathbf 1}$ &
        $\displaystyle{\frac{p + e_{23}}{2}}$ &\\
\hline \rule[-2mm]{0mm}{8mm}
  &\textnormal{condition}   & $\ord_p(Q_N(\sigma)) \geq 3$ & $Q_N(\sigma) \sim p^2\epsilon_1$ & $Q_N(\sigma) \sim p^2\Delta\epsilon_1$ \\ \cline{2-5}
\raisebox{2.5ex}[0cm][0cm]{\bf(7)}& $\vert \Lpinv(N)_\sigma\vert$ & $\displaystyle{\frac{p - e_{12}}{2}}$ & $ p\,\eta_{12} + {\mathbf 1}$ &
            $ p\,\eta'_{12}$\\
\hline \rule[-2mm]{0mm}{8mm}
  &\textnormal{condition}   & $\ord_p(Q_N(\sigma)) \geq 3$ & $\ord_p(Q_N(\sigma)) = 2$ & \\ \cline{2-4}
\raisebox{2.5ex}[0cm][0cm]{\bf(8)}& $\vert \Lpinv(N)_\sigma\vert$ & $p$ & {\bf 1} & \\
\hline
\end{tabular}}
\vskip 0.2cm
\center{\rm Table II}
\end{table}
\end{prop}
\begin{proof}
We will provide the detail only for Case {\bf (4)} since the arguments used in this case can be applied to prove the other cases.  So, $L_p \cong \langle \epsilon_1, p\epsilon_2, p^2\epsilon_3 \rangle$, and $N_p \cong \langle p^2\epsilon_1, p \epsilon_2, p^2 \epsilon_3\rangle$.  Let $\sigma = \tau_w$ be a symmetry in $O(N)$ with $w$ a primitive vector in $N$.  Thus, $Q_N(\sigma) = Q(w)$, which has only the three possibilities listed in Table II.  By Lemma \ref{trivial}, $N_p = \z_p w \perp G_p$ where $G$ is the orthogonal complement of $w$ in $N$.  In below, for any  $a \in \z_p$ we write $\overline{a}$ to denote the canonical image of $a$ in $\z_p/p\z_p \cong \mathbb F_p$.

Suppose that $Q_N(\sigma) \sim p\epsilon_2$.  Then $\z_pw \cong \langle p\epsilon_2 \rangle$ and $G_p \cong \langle p^2\epsilon_1, p^2\epsilon_3\rangle$.  The number of $M \in \Lpinv(N)_\sigma$ is then equal to half of the number of representations of $\overline{\epsilon_1}$ by the quadratic space $\langle \overline{\epsilon_1}, \overline{\epsilon_3}\rangle$ over $\mathbb F_p$, which is $\frac{p - e_{13}}{2}$ by \cite[Lemma1.3.2]{ki}.

Now, suppose that $Q_N(\sigma) \sim p^2\epsilon_1$.  It is clear that there is exactly one lattice in $\Lpinv(L)_\sigma$ to which $\sigma$ is special, namely the one whose $p$-adic completion is $\z_p(p^{-1}w) \perp G_p$.  If $\sigma$ is not special to an $M \in \Lpinv(N)_\sigma$, then $\z_p w \cong \langle p^2\epsilon_1 \rangle$ would be a Jordan component of $L_p$ and hence $\epsilon_1 \sim \epsilon_3$.   If that is the case, then the number of lattices in $\Lpinv(N)_\sigma$ to which $\sigma$ is not special is equal to the number of over-lattices of $\langle p\epsilon_2, p^2\epsilon_3 \rangle$ that are isometric to $\langle p\epsilon_2, \epsilon_3 \rangle$.  This number is 1 because  $\langle p\epsilon_2, \epsilon_3 \rangle$ is a $\z_p$-maximal lattice on an anisotropic quadratic space over $\q_p$.

Finally, let us assume that $Q_N(\sigma) \sim p^2\Delta\epsilon_1$.  It is easy to see that if $\Delta\epsilon_1 \not \sim \epsilon_3$, then $\vert \Lpinv(N)_\sigma \vert = 0$.  Suppose that $\Delta\epsilon_1 \sim \epsilon_3$.  Then $\vert \Lpinv(N)_\sigma \vert$ is  equal to the number of over-lattices of $\langle p^2\epsilon_1, p\epsilon_2 \rangle$ that are isometric to $\langle \epsilon_1, p\epsilon_2\rangle$, which is 1 as is explained in the last paragraph.
\end{proof}

\subsection{Another equation}

Suppose that $N = \Lp(L)$.  We define $s = s(N,L)$ to be the number of symmetry $\sigma \in O(N)$ with a lattice $M \in \Lpinv(N)_\sigma$ to which $\sigma$ is special.  The next corollary gives us another equation involving the $h_{2d}(N)$'s.

\begin{prop} \label{8fin}
Suppose that $N = \Lp(L)$, and that $p \neq 3$ when $\vert O(N) \vert = 24$.  If $\vert O(N)\vert$ is divisible by $8$, then
\begin{equation} \label{s}
s=\begin{cases}
h_8(N) &\mbox{ if $\vert O(N)\vert = 8$},\\
h_{16}(N)+2h_8(N)  &\mbox{ if $\vert O(N)\vert =16$}, \\
h_{24}(N)+3h_8(N)  &\mbox{ if $\vert O(N)\vert = 24$},\\
3h_{16}(N) + 6h_8(N) & \mbox{ if $\vert O(N) \vert = 48$}.
\end{cases}
\end{equation}
Furthermore, suppose that $\sigma$ is a symmetry in $O(N)$, $48 > \vert O(N)\vert >8$, and $M \in \Lpinv(N)_\sigma$ to which $\sigma$ is special.  Then $O(M)=O(N)$ if and only if $\sigma$ is contained in the center of $O(N)$.
\end{prop}
\begin{proof}
Let $\sigma$ be a symmetry in $O(N)$, and suppose that $M \in \Lpinv(L)_\sigma$ to which $\sigma$ is special.  Then $\vert O(M) \vert$ is divisible by 8 by Proposition \ref{order8}.  So, it follows from Section \ref{isometry} that either $\vert O(M) \vert = 8$ or $M = K_i(a,b)$ for some $i \in \{2,3,4\}$ and positive integers $a, b$.   Note that $\vert O(M) \vert$ cannot be 48.  If $\vert O(M) \vert = 8$, then $O(M)$ has exactly three symmetries $\tau_w, \tau_u, \tau_v$, and $w, u, v$ are mutually orthogonal.  So, exactly one of these three symmetries is special to $M$.  Suppose that $M = K_i(a,b)$.  Following the notations used in Section \ref{isometry}, one can check that $\tau_{x_3}$ for $K_2(a,b)$ or $K_3(a,b)$ and $\tau_{x_1 + x_2 + 2x_3}$ for $K_4(a,b)$ is the only symmetry that is special to $M$ (note that this requires $p \neq 3$ for $K_2(a,b)$). This symmetry is the only symmetry in the center of $O(M)$.  This proves (\ref{s}) and the ``only if" part of the second assertion.

For the ``if" part of the second assertion, note that the proof of Proposition \ref{order8} actually shows that any symmetry in $O(N)$ which commutes with $\sigma$ is in $O(M)$.  So, if $\sigma$ is in the center of $O(N)$, then $O(M)$ contains all the symmetries in $O(N)$, and therefore $O(M)$ is equal to $O(N)$.
\end{proof}

\begin{rmk}\label{uniquespecial}
Embedding in the statement (and the proof) of Proposition \ref{8fin} is the fact that if $M$ has a special symmetry, then that symmetry is the only special symmetry of $M$.
\end{rmk}

The value of $s$ depends on the Jordan decomposition of $L_p$ and $Q_N(\sigma)$ for all symmetries $\sigma$ of $N$.  Suppose that $L_p \cong \langle \epsilon_1, p^\alpha \epsilon_2, p^\beta \epsilon_3 \rangle$ as in Table I, and $\epsilon_i \in \z_p^\times$ for $i = 1,2 ,3$.
Define an integer $t$ and $\epsilon \in \z_p^\times$ by
\begin{equation} \label{t}
(t, \epsilon) = \begin{cases}
(\beta, \epsilon_3) & \mbox{ when $\alpha = 0$},\\
(2, \epsilon_1) & \mbox{ otherwise},
\end{cases}
\end{equation}
and let
$$\mathfrak S(N,L) = \{\sigma \in S(N) : \ord_p(Q_N(\sigma)) = t \mbox{ and } p^{-t}Q_N(\sigma)\epsilon \in (\z_p^\times)^2\}.$$

\begin{prop} \label{svalue}
Suppose that $N = \Lp(L)$ and $\vert O(N)\vert$ is divisible by $8$.  Then $s$ is the cardinality of the set $\mathfrak S(N,L)$.
\end{prop}
\begin{proof}
When $\alpha = 0$, $N_p \cong \langle p^2\epsilon_1, p^2\epsilon_2, p^\beta \epsilon_3 \rangle$.  Let $\tau_w$ be a symmetry of $O(N)$ with $w$ primitive in $N$.  By Lemma \ref{sym}, $O(N)$ contains two other symmetries $\tau_u, \tau_v$ such that $N_p = \z_p w \perp \z_pu \perp \z_p v$. Thus $\tau_w$ is special to a lattice $M \in \Lpinv(N)$ if and only if $\ord_p(Q(w)) = \beta$ and $p^{-\beta}Q(w)$ and $\epsilon_3$ are in the same square class.  This proves the proposition when $\alpha = 0$.  The remaining cases can be proved by similar consideration.
\end{proof}

\section{Class numbers}

\begin{defn} \label{label}
Let $K$ be a ternary lattice, and $\sigma_1, \ldots, \sigma_t$ be all the symmetries of $K$ arranged so that $Q_K(\sigma_i) \leq Q_K(\sigma_{i+1})$ for $i = 1, \ldots, t-1$.  The {\em label} of $K$ is defined as
$$\lab(K): = \lb \vert O(K) \vert; Q_K(\sigma_1), \ldots, Q_K(\sigma_t)\rb.$$
\end{defn}

For example, if $K$ has trivial isometry group, then $\lab(K) = \lb 2 \rb$.  For $K_2(a,b)$ when $b > 6a$, its label is $\lb 24; 2a, 2a, 2a, 6a, 6a, 6a, b \rb$.

Let $L$ be a primitive ternary lattice and $N = \Lp(L)$.  We define
$$w = \vert \Lpinv(N) \vert \mbox{ and }  f = \sum_{\sigma \in S(N)} \vert \Lpinv(N)_\sigma \vert.$$
Furthermore, for any positive integer $n$, we let $[w]_n$ be the remainder of $w$ when divided by $n$.  The values of $w$, $f$, and $s$ can be effectively determined by Table I, Table II, and Proposition \ref{svalue} respectively, using only the label of $N$ and a Jordan decomposition of $L_p$.  In below, the order of an isometry $\sigma$ is denoted by $o(\sigma)$.

\begin{thm}\label{class}
Suppose that $\Lp(L) = N$ and $p \neq 3$ when $\vert O(N) \vert = 24$.  Then for each positive divisor $d$ of $\vert O^+(N) \vert$, $h_{2d}(N)$, the number of classes in $\Lpinv(N)$ with isometry group of size $2d$, is determined by the label of $N$ as shown in Table III.
\begin{table} [ht]
\centering
{\tabcolsep=3pt
\begin{tabular}{|c||c|c|c|c|c|c|}
\hline \rule[-2mm]{0mm}{8mm}
 $\vert O(N) \vert$ & $h_2(N)$ & $h_4(N)$ & $h_8(N)$ & $h_{12}(N)$ & $h_{16}(N)$ & $h_{24}(N)$ \\
\hline \rule[-2mm]{0mm}{8mm}
$2$ & $w$ &$0$ & $0$& $0$& $0$& $0$ \\
\hline \rule[-2mm]{0mm}{8mm}
$4$ & $\frac{w-f}{2}$ & f &$0$ &$0$ & $0$& $0$\\
\hline \rule[-2mm]{0mm}{8mm}
$8$ & $\frac{w - f + 2s}{4}$ & $\frac{f - 3s}{2}$ & $s$ &$0$ & $0$&$0$ \\
\hline \rule[-2mm]{0mm}{8mm}
$12$ & $\frac{w - f + 2[w]_3}{6}$ & $\frac{f - 3[w]_3}{3}$ &$0$ & $[w]_3$ &$0$ &$0$ \\
\hline \rule[-2mm]{0mm}{8mm}
$16$ & $\frac{w - f + 2s + 2[w]_2}{8}$ & $\frac{f - 3s - 2[w]_2}{4}$ & $\frac{s - [w]_2}{2}$ &$0$ & $[w]_2$ &$0$ \\
\hline \rule[-2mm]{0mm}{8mm}
$24$ & $\frac{w - f + 2s + 4[w]_3}{12}$ & $\frac{f - 3s - 4[w]_3}{6}$ & $\frac{s - [w]_3}{3}$ & $0$ & $0$ & $[w]_3$\\
\hline
\end{tabular}}
\vskip 0.2cm
\center{\rm Table III}
\end{table}

When $\vert O(N) \vert = 48$, $h_{24}(N) = h_{48}(N) = 0$, $h_{12}(N)$ and $h_{16}(N)$ are determined by (\ref{48}), $h_8(N) = \frac{s - 3h_{16}(N)}{6}$,
\begin{eqnarray*}
h_4(N) & = & \frac{f - 18h_8(N) -12h_{12}(N) - 15h_{16}(N)}{12};\\
h_2(N) & = &  \frac{w - f + 12h_8(N) + 8h_{12}(N) + 12h_{16}(N)}{24}.
\end{eqnarray*}
\end{thm}
\begin{proof}
We will provide the proofs for the cases $\vert O(N) \vert = 24$ and $\vert O(N) \vert = 48$; the other cases are easier and can be proved by the same argument.

Suppose that $\vert O(N) \vert = 24$.  By Proposition \ref{order3}, we know that $h_{12}(N) = 0$ and $h_{24}(N) \leq 1$.  Also, if $M$ is a ternary lattice whose isometry group has order 24, then $O^+(M)\cong \z_2\oplus D_3$ contains two elements of order 3 and two elements of order 6.  This implies
$$\sum_{o(\sigma) = 3, 6} \vert \Lpinv(N)_\sigma \vert = 4h_{24}(N).$$
Therefore, by (\ref{h1}), (\ref{first}), and (\ref{s}), we have a system of three equations
\begin{eqnarray*}
h_2(N) + h_4(N) + h_8(N) + h_{24}(N) & = & \frac{1}{12}(w + f + 4h_{24}(N))\\
12h_2(N) + 6h_4(N) + 3h_8(N) + h_{24}(N) & = & w\\
3h_8(N) + h_{24}(N) & = & s.
\end{eqnarray*}
The second equation shows that $h_{24}(N) = [w]_3$, and by the third we obtain $h_8(N) = \frac{s - [w]_3}{3}$.  Substituting these back  into the first two equations results in a system of two linear equations in $h_2(N)$ and $h_4(N)$ which has the unique solution as presented in Table III.

Now suppose that $\vert O(N) \vert = 48$.    Again, we may assume that $N = p\bI$.  The values of $h_{12}(N)$ and $h_{16}(N)$ are determined by (\ref{48}); both are either 0 or 1.  By (\ref{s}), we have $s = 6h_8(N) + 3h_{16}(N)$, and hence $h_8(N) = \frac{s - 3h_{16}(N)}{6}$.  Equation (\ref{first}) leads us to the equation
$$h_2(N) + \frac{h_4(N)}{2} = \frac{1}{24}(w - 6h_8(N) - 4h_{12}(N) - 3h_{16}(N)).$$
If $M \in \Lpinv(N)$ has an isometry group of order 12, then $O^+(M)$ contains exactly two isometries  of order 3.  So,
$$\sum_{o(\sigma) = 3} \vert \Lpinv(N)_\sigma \vert = 8h_{12}(N).$$
Similarly, if $\vert O(M) \vert = 16$, then $O^+(M)$ contains exactly two isometries of order 4, which means
$$\sum_{o(\sigma) = 4} \vert \Lpinv(N)_\sigma \vert = 6h_{16}(N).$$
Putting everything in (\ref{h1}) gives us another equation
$$h_2(N) + h_4(N) = \frac{1}{24}(w + f - 24h_8(N) - 16 h_{12}(N) - 18h_{16}(N)).$$
So, now we have two linear equations in the unknowns $h_2(N)$ and $h_4(N)$.  One can check easily that the unique common solution to these two equations is the one in Table III.
\end{proof}

\begin{rmk} \label{remarkI}
Suppose that $N = p\bI$.  It has three symmetries $\sigma$ with $Q_N(\sigma) = p^2$, and six symmetries $\sigma$ with $Q_N(\sigma) = 2p^2$.  Therefore,  $s$ is 0, 3, 6, or 9.  Since $h_{16}(N) = 0$ or 1, we must have
$$h_8(N) = \frac{s - 3h_{16}(N)}{6} = \begin{cases}
0 & \mbox{ when $(s, h_{16}(N)) = (0,0)$ or $(3,1)$},\\
1 & \mbox{ when $(s, h_{16}(N)) = (6,0)$ or $(9,1)$}.
\end{cases}$$
\end{rmk}

\begin{exam}\label{I}
Let $p > 3$ and $L$ be a lattice of discriminant $p^2$ such that $L_p \cong \langle 1, 1, p^2 \rangle$.  Therefore, $\Lp(L) = p\bI$ which has class number one.  Since $\vert O(\bI) \vert = 48$,  it follows from (\ref{48}) that
$$h_{16}(N) = 1 \mbox{ and } h_{12}(N) = \frac{1 + \left(\frac{3}{p} \right)}{2}.$$
The three quantities $s$, $w$, and $f$ are easy to obtain, since we know the symmetries in $O(p\bI)$ well.  There are three symmetries $\sigma$ of $p\bI$ with $Q_{p\bI}(\sigma) = p^2$, and another six symmetries $\tau$ with $Q_{p\bI}(\tau) = 2p^2$.  Therefore,
$$s = 3 + 3\left (1 + \left(\frac{2}{p} \right)\right),$$
and thus
$$h_8(N) = \frac{1 + \left(\frac{2}{p} \right)}{2}.$$
From Table I and Table II, the values of $w$ and $f$ are given by
$$w = \frac{p\left(p + \left(\frac{-1}{p}\right)\right)}{2},$$
and
$$f = 9\left (\frac{p - \left(\frac{-1}{p}\right)}{2} + 1 \right) + 6\left( \left(\frac{-1}{p}\right) - 1\right) \left( \frac{1 - \left(\frac{2}{p}\right)}{2} \right)$$
respectively.  The exact values of $h_2(N)$ and $h_4(N)$ can be determined by Table III.  Adding all the $h_{2d}(N)$ together we have the class number of $L$ as
$$h(L) = \frac{1}{48}\left ( p^2 + p\left(9 + \left(\frac{-1}{p}\right) \right) - 3\left(\frac{-1}{p}\right) + 6\left(\frac{2}{p}\right) - 6\left(\frac{-2}{p}\right) + 8\left(\frac{3}{p}\right) + 32\right).$$
\end{exam}

\begin{exam} \label{example}
Let $L$ be a ternary lattice with the Gram matrix
$$\begin{pmatrix} 2&0&-p\\0&2&-p\\-p&-p&7p^2\end{pmatrix},$$
where $p > 3$.  It is easy to see that $L_p \cong \langle 2, 2, 6p^2\rangle$ and $N  = \Lp(L)$ is the lattice $K_4(p^2, 7p^2)$; in particular, $K_4(1, 7)$ is the primitive lattice $\lp(L)$.  It is known that $h(N) = 1$, and by Section \ref{isometry} the label of $N$ is $\lb 16; 2p^2, 2p^2, 4p^2, 4p^2, 24p^2 \rb$.  To simplify the discussion, let us further assume that $p \equiv 7$ mod 24, which means that 2 is a square in $\z_p$ but 3 and $-1$ are not.  Then $s = 1$, $w = \frac{p(p-1)}{2}$, and $f = 5(\frac{p-1}{2}) + 2$.  As a result, $h_{16}(N) = 1$, $h_8(N) = 0$,
$$h_4(N) = \frac{5p - 11}{8}, \mbox{ and } h_2(N) = \frac{p^2 - 6p + 9}{18}.$$
Adding all these $h_{2d}(N)$ together yields
$$h(L) = \frac{p^2 + 4p + 3}{16}.$$
Similar calculations show that
$$h(L)=\begin{cases}
\displaystyle{\frac{p^2+6p+9}{16}} \quad &\text{if $p \equiv 1,5,13,17 \mod{24}$,}\\
\displaystyle{\frac{p^2+4p+3}{16}} \quad &\text{if $p \equiv 7 \mod{24}$,}\\
\displaystyle{\frac{p^2+6p+11}{16}} \quad &\text{if $p \equiv 11,19 \mod{24},$}\\
\displaystyle{\frac{p^2+4p+19}{16}} \quad &\text{if $p \equiv 23 \mod{24}.$}\\
\end{cases}$$
\end{exam}

\section{Labels of Classes} \label{changeoflabels}

Suppose that $\Lp(L) = N$.  We have seen in Theorem \ref{class} that the class number of $L$ is determined only by the label of every class in $\gen(N)$ and a Jordan decomposition of $L_p$.  In order to apply that theorem successively, we need to show that the labels of all the classes in $\Lpinv(N)$ are also determined by the label of $N$ and the structure of $L_p$.  For each class of lattices in $\Lpinv(N)$, we define its label to be the label of any one of its lattices.  The label of $\Lpinv(N)$ is defined to be the multi-set which contains all the labels of classes in $\Lpinv(N)$.    More generally, for a subset $\mathfrak X$ of $\Lpinv(N)$, we define the label of $\mathfrak X$ to the multi-set containing all the labels of classes of lattices in $\mathfrak X$.

\subsection{Preliminary lemmas}

Let $\sigma$ be a symmetry in $O(N)$, and suppose that $M \in \Lpinv(N)$.  Then $Q_M(\sigma)$ is determined by $Q_N(\sigma)$ as shown in Table IV below; this  follows directly from the definition of special symmetry.  The cases {\bf (1)} to {\bf (8)} are divided as in Table I.
\begin{table} [ht]
\begin{centerline}{\tabcolsep=3pt
\begin{tabular}{|c||c|c|}
\hline \rule[-2mm]{0mm}{8mm}
{}&{\bf (1)} and {\bf (2)}&{\bf (3)} to {\bf (8)} \\
\hline \hline
\rule[-2mm]{0mm}{6mm} $\sigma$ : special to $M$ & $Q_N(\sigma)$ & $\frac{1}{p^2}Q_N(\sigma)$\\
\hline
\rule[-2mm]{0mm}{6mm} $\sigma$ : not special to $M$ & $\frac{1}{p^2}Q_N(\sigma)$ & $Q_N(\sigma)$\\
\hline
\end{tabular}}
\end{centerline}
\vskip 0.2cm
\center{\rm Table IV \quad Value of $Q_M(\sigma)$}
\end{table}

For any positive integer $d > 1$, let $\gH_{2d}$ be the set of lattices $M \in \Lpinv(N)$ whose isometry group have order $2d$.  For each symmetry $\sigma$ in $O(N)$, let $\gH_{2d}(\sigma)$ be the set $\gH_{2d} \cap \Lpinv(N)_\sigma$, that is, the set containing all lattices $M \in \Lpinv(N)$ such that $\sigma \in O(M)$ and $\vert O(M)\vert = 2d$. Clearly,
\begin{equation} \label{union}
\gH_{2d} = \bigcup_{\sigma \in S(N)} \gH_{2d}(\sigma).
\end{equation}
The number of classes in $\gH_{2d}(\sigma)$ is denoted by $h_{2d}(\sigma)$.  In general, (\ref{union}) may not be a disjoint union.

\begin{lem}\label{conjugacy}
Let $M$ and $M'$ be two lattices in $\Lpinv(N)$.  If $M \cong M'$, then $O(M)$ and $O(M')$ are conjugate inside $O(N)$.
\end{lem}
\begin{proof}
Suppose that $\phi: M \longrightarrow M'$ is an isometry.  Then $\phi$ is necessarily an isometry of $N$, by Lemma \ref{conf}.  In that case, $O(M') = \phi O(M) \phi^{-1}$. This proves the lemma.
\end{proof}

\begin{lem} \label{conjugacy2}
Let $\sigma$ and $\sigma'$ be two symmetries of $N$.
\begin{enumerate}
\item[(a)] If $\phi \sigma \phi^{-1} = \sigma'$ for some $\phi \in O(N)$, then for all $d$, $\phi$ induces a bijection from $\gH_{2d}(\sigma)$ onto $\gH_{2d}(\sigma')$.  Consequently, the classes in $\gH_{2d}(\sigma)$ coincide with the classes in $\gH_{2d}(\sigma')$.

\item[(b)] If there is an isometry from a lattice $M \in \gH_4(\sigma)$ to another lattice $M' \in \gH_4(\sigma')$, then $\sigma$ and $\sigma'$ are conjugate in $O(N)$.
\end{enumerate}
\end{lem}
\begin{proof}
For part (a), it is obvious that $M \longmapsto \phi(M)$ is an injective function from $\gH_{2d}(\sigma)$ to $\gH_{2d}(\sigma')$.  It has an inverse, which is induced by $\phi^{-1}$.  For the other assertion, note that if $M_1, M_2 \in \gH_{2d}(\sigma)$ and $\psi: M_1 \longrightarrow M_2$ is an isometry, then $\psi \phi \psi^{-1}$ is an isometry from $\phi(M_1)$ to $\phi(M_2)$.

For part (b), suppose that $\phi: M\longrightarrow M'$ is an isometry.  Then $\phi$ is an isometry on $O(N)$, by virtue of Lemma \ref{conf}.  It is straightforward to see that $\phi \sigma \phi^{-1}$ is a symmetry in $O(M')$.  But since $M' \in \gH_4(\sigma')$, it has one and only one symmetry, namely $\sigma'$.  Thus $\sigma' = \phi \sigma \phi^{-1}$.
\end{proof}

\begin{cor}\label{conjugacyh4}
Let $C$ be a complete set of representatives of conjugacy classes of symmetries of $O(N)$.  Then the label of $\gH_4$ is the (multi-set) union of the labels of $\gH_4(\sigma)$, for all $\sigma \in C$.
\end{cor}

\begin{lem} \label{conjugacy16h8}
Suppose that $\vert O(N) \vert = 16$.  Let $\tau$ be the symmetry in the center of $O(N)$, and $\sigma$ and $\sigma'$ be two symmetries of $N$, which are not $\tau$ and are not conjugate in $O(N)$.  Then the label of $\gH_8$ is the disjoint union of the labels of $\gH_8(\sigma)$ and $\gH_8(\sigma')$, and it is also equal to the label of $\gH_8(\tau)$.
\end{lem}
\begin{proof}
Since $\vert O(N) \vert = 16$, $N$ has exactly two orthogonal systems; see Section \ref{isometry}.  If $M$ is a lattice in $\gH_8$, then $S(M)$ is one of the orthogonal systems.  Since $\sigma$ and $\sigma'$ are not conjugate in $O(N)$, they belong to different orthogonal systems (see Section \ref{isometry}), and hence the label of $\gH_8$ is the disjoint union of the labels of $\gH_8(\sigma)$ and  $\gH_8(\sigma')$.

The last assertion is clear since $\tau$ is in both orthogonal systems, and hence $\gH_8$ is just $\gH_8(\tau)$.
\end{proof}

\begin{lem}\label{conjugacy24h8}
Suppose that $\vert O(N) \vert = 24$.  Then the label of $\gH_8$ is equal to the label of $\gH_8(\sigma)$ for any $\sigma \in S(N)$.
\end{lem}
\begin{proof}
From Section \ref{isometry}, $S(N)$ is decomposed into three conjugacy classes.  There are three orthogonal systems in $O(N)$, each of them is of the form $\{\sigma, \sigma', \tau\}$, where $\sigma$ and $\sigma'$ belong to different conjugacy classes in $O(N)$ and $\tau$ is the unique symmetry in the center of $O(N)$.  Since each symmetry other than $\tau$ is contained in one and only one orthogonal system, all three orthogonal systems of $O(N)$ are conjugate.  Therefore, if $\mathcal O$ is an orthogonal system and $\gH_8(\mathcal O)$ denote the set of lattices $M \in \gH_8$ such that $S(M) = \mathcal O$, then the label of $\gH_8$ is equal to the label of $\gH_8(\mathcal O)$.

The symmetry $\tau$ is in all three orthogonal systems.  Therefore, $\gH_8$ is just $\gH_8(\tau)$, and hence their labels are the same.  Now, let $\sigma \in S(N)$ which is not $\tau$, and $\mathcal O$ be the unique orthogonal system containing $\sigma$.  Then $\gH_8(\sigma)$ must be equal to $\gH_8(\mathcal O)$, and so the label of $\gH_8(\sigma)$ is just the label of $\gH_8$.
\end{proof}

For any $\sigma \in S(N)$ and any $M \in \Lpinv(N)_\sigma$, let
$$G_M(\sigma) = \{\phi \in O(N) : \sigma \in O(\phi(M)) \}.$$
Although $G_M(\sigma)$ is not necessarily a subgroup of $O(N)$, it contains the normalizer of $O(M)$.  Moreover, the size of the coset space $G_M(\sigma)/O(M)$, denoted $g_M(\sigma)$, is the number of lattices in $\text{cls}(M) \cap \Lpinv(N)_\sigma$.  If $g_M(\sigma)$ is the same for every $M$ in $\gH_{2d}(\sigma)$, then $\vert \gH_{2d}(\sigma) \vert = g_{M}(\sigma)h_{2d}(\sigma)$.

\begin{lem}\label{h4}
Let $\sigma \in S(N)$ and $M$ be a lattice in $\gH_4(\sigma)$.  Then $g_M(\sigma)$ is the group index $[C(\sigma) : \{\pm I, \pm \sigma\}]$, where $C(\sigma)$ is the centralizer of $\sigma$ in $O(N)$.
\end{lem}
\begin{proof}
Clearly, $O(M)$ is $\{\pm I, \pm \sigma\}$ which contains a unique symmetry, namely $\sigma$.  For any $\phi \in O(N)$, $\phi^{-1}\sigma\phi$ is also a symmetry.  Therefore, $\phi \in G_M(\sigma)$ if and only if $\phi^{-1}\sigma \phi = \sigma$, that is $\phi \in C(\sigma)$.
\end{proof}

\subsection{Main Theorem}

\begin{thm} \label{labelchange}
Suppose that $\Lp(L) = N$.  Then the label of $\Lpinv(N)$ can be computed effectively by using the label of $N$ and a Jordan decomposition of $L_p$.
\end{thm}
\begin{proof}
The proof is a case-by-case analysis according to the size of $O(N)$.  For each case, the proof will show how we can determine the label of $\Lpinv(N)$.  For simplicity, we write $h_{2d}$ for $h_{2d}(N)$ and $\Gamma$ for $\Lpinv(N)$ in the following discussion.  We recall that all the numbers $h_{2d}$ can be obtained from Theorem \ref{class}.   To determine the label of $\Gamma$, it suffices to determine the label of $\gH_{2d}$ for each possible $d$.  For $\gH_4$, it suffices to determine the label of $\gH_4(\sigma)$, where $\sigma$ runs through a complete set of representatives of conjugacy classes of symmetries in $O(N)$.  Note that these conjugacy classes are explicitly described in Section \ref{isometry}.
\vskip 2mm

\noindent \fbox{$\vert O(N) \vert = 4$}:  For the unique $\sigma \in S(N)$, $h_4(\sigma)$ is simply $\vert \Gamma_\sigma \vert$, which can be determined by Table II.  By Proposition \ref{specialone}, there is at most one lattice in $\Gamma_\sigma$ to which $\sigma$ is special, and whether or not such a lattice exists is determined by Table II.  Therefore, the label of $\gH_4 = \gH_4(\sigma)$ can be determined.
\vskip 2mm

\noindent \fbox{$\vert O(N) \vert = 8$}: In this case, $O(N)$ contains exactly three symmetries, and these symmetries commute with each other. Thus each symmetry forms its own conjugacy class in $O(N)$.   Clearly, $h_8(\sigma)$ is just $h_8$ for each $\sigma \in S(N)$.  If $\sigma$ is special to some $M \in \Gamma_\sigma$, then $M \in \gH_8(\sigma)$ by Proposition \ref{order8}.  Moreover $\sigma$ will be the only symmetry special to $M$.  Using Table II we can determine which $\sigma$ is special to $M$.   Therefore, the label of $\gH_8$ is determined.

If $M \in \gH_4(\sigma)$, then $G_M(\sigma) = O(N)$ because $O(M)$ is normal in $O(N)$.  Therefore,
$$\vert \Gamma_\sigma \vert = 2 h_4(\sigma) + h_8(\sigma),$$
and hence $h_4(\sigma)$ can be determined.  Since $\sigma$ is not special to any lattice in $\gH_4(\sigma)$, the label of $\gH_4(\sigma)$ is determined.
\vskip 2mm

\noindent \fbox{$\vert O(N)\vert = 12$}: First of all, $h_{12}$, which is either 0 or 1, is known.  Moreover, none of the symmetries in $O(N)$ is special to any $M \in \gH_{12}$.  Therefore, the label of $\gH_{12}$ can be computed.

Let $\sigma$ be a symmetry in $O(N)$.  Then $C(\sigma)$ is $\{\pm I, \pm \sigma \}$.  Therefore,
$$\vert \Gamma_\sigma \vert = h_4(\sigma) + h_{12},$$
and so $h_4(\sigma)$ is determined.  From Table II we can decide if $\sigma$ is special to any lattice in $\Gamma_\sigma$.  Therefore, the label of $\gH_4(\sigma)$ can be determined.
\vskip 2mm

\noindent \fbox{$\vert O(N) \vert = 16$}:  We denote the unique symmetry in the center of $O(N)$ by $\tau$.  Let $M$ be a lattice in $\Gamma$.  Then, by Proposition \ref{8fin}, $\vert O(M) \vert = 16$ if and only if $\tau$ is special to $M$.  This shows that the label of $\gH_{16}$ is determined.

For the label of $\gH_8$, it suffices to compute the label of $\gH_8(\sigma)$ for any symmetry $\sigma$ of $N$ which is not in the center.  Let $\{\sigma, \sigma', \tau\}$ be the orthogonal system that contains $\sigma$.    From the description of orthogonal systems and conjugacy classes of symmetries in Section \ref{isometry}, we see that $\sigma$ and $\sigma'$ are conjugate in $O(N)$.  Therefore, by Propositions \ref{specialone} and \ref{order8}, $h_8(\sigma)$ is either 0 or 1, and it is 1 if and only if $\sigma$ is special to a lattice in $\gH_8(\sigma)$.   This can be determined by computing $\vert \Gamma_{\sigma} \vert$ using Table II.  So, the label of $\gH_8$ can be determined.  Note also that $O(M)$ is normal in $O(N)$ whenever $M \in \gH_8$.  Therefore, $\vert \gH_8(\sigma)\vert = 2h_8(\sigma)$.

Suppose that $M \in \gH_4$.  Then $M \in \gH_4(\sigma)$ for some $\sigma \in S(N)$, and $\sigma$ is not special to $M$ by Proposition \ref{order8}.  Therefore, the label of $M$ is determined.  It remains to compute $h_4(\sigma)$ for every $\sigma$ in $S(N)$.  Suppose that $\sigma$ is not $\tau$.  Then $C(\sigma)$ is $\{\pm I\}\times \{1, \tau, \sigma, \sigma\tau\}$ which has order 8.  Therefore,
\begin{equation} \label{order1641}
\vert \Gamma_\sigma \vert = 2h_4(\sigma) + 2h_8(\sigma) + h_{16},
\end{equation}
which shows that $h_4(\sigma)$ can be determined.  One the other hand, for the symmetry $\tau$, we have the equation
\begin{equation} \label{order1642}
\vert \Gamma_\tau \vert = 4h_4(\tau) + 2h_8(\tau) + h_{16}
\end{equation}
because $C(\tau) = O(N)$.  Since $\tau$ is in $O(M)$ whenever $M \in \gH_8$, therefore $h_8(\tau) = h_8$.  This shows that $h_4(\tau)$ is also determined.
\vskip 2mm

\noindent \fbox{$\vert O(N) \vert = 24$}:  Let $\tau$ be the unique symmetry in the center of $O(N)$.  Whether or not $\tau$ is special to any lattice can be determined by Table II, and by Proposition \ref{8fin} it is indeed special to some lattice $M \in \Gamma$ if and only if $M \in \gH_{24}$.   Since $h_{24}$ is 0 or 1, the label of $\gH_{24}$ is determined.

Let $\sigma \in S(N)$ which is not $\tau$, and let $\{\sigma, \sigma', \tau\}$ be the orthogonal system containing $\sigma$.  Again, from the description of orthogonal systems in Section \ref{isometry}, we see that $\sigma$ and $\sigma'$ belong to different conjugacy classes in $O(N)$.   It follows from Propositions \ref{specialone} and \ref{order8} that every lattice in $\gH_8(\sigma)$ must have either $\sigma$ or $\sigma'$ as its unique special symmetry.   Therefore, the label of each lattice in $\gH_8(\sigma)$ can be determined once we know whether $\sigma$ or $\sigma'$ is special to any lattice.  The latter can be checked by using Table II.  So, the label of $\gH_8(\sigma)$, and hence the label of $\gH_8$, can be determined.   Furthermore, $C(\sigma)$ is  $\{\pm I\}\times \{I, \sigma, \tau, \tau\sigma\}$, which has order 8.  Therefore, $\vert \gH_8(\sigma) \vert = h_8(\sigma)$,
\begin{equation}
\vert \Gamma_\sigma \vert = 2h_4(\sigma) + h_8(\sigma) + h_{24},\label{24first}
\end{equation}
and from this equation we can determine $h_4(\sigma)$.  Since $\sigma$ is not special to any lattice in $\gH_4(\sigma)$, the label of $\gH_4(\sigma)$ is determined.

For the symmetry $\tau$, $C(\tau)$ is $O(N)$ because $\tau$ is in the center of $O(N)$.  Therefore, $\vert \gH_8(\tau) \vert = 3h_8(\tau) = 3h_8$ and $\vert \gH_4(\tau) \vert = 6h_4(\tau)$; hence
\begin{equation}
\vert \Gamma_\tau \vert = 6h_4(\tau) + 3h_8(\tau) + h_{24},\label{24second}
\end{equation}
and $h_4(\tau)$ can be determined.  Since $\tau$ is not special to any lattice in $\gH_4$, the label of $\gH_4(\tau)$ is determined.
\vskip 2mm

\noindent \fbox{$\vert O(N) \vert = 48$}: We only present the argument for the case $N = p\bI$.  Let $\{x_1, x_2, x_3\}$ be an orthogonal basis of $N$.  We distinguish the symmetries of $N$ into two types:
$$\begin{cases}
\mbox{ Type I }: & \tau_{x_1}, \tau_{x_2}, \tau_{x_3},\\
\mbox{ Type II }: & \tau_{x_i + x_j} \mbox{ and } \tau_{x_i - x_j}, 1\leq i < j \leq 3.
\end{cases}$$
So, $Q_N(\sigma) = p^2$ if $\sigma$ is a Type I symmetry; otherwise $Q_N(\sigma) = 2p^2$.

By Theorem \ref{class}, $h_{24} = h_{48} = 0$, and both $h_{12}$ and $h_{16}$ are either 0 or 1.   The lattices in $\gH_{12}$ and $\gH_{16}$, if they exist, are isometric to
$$\begin{cases}
K_1(1, \frac{p^2+2}{3})\, (\mbox{disc $= p^2$})  \mbox{ or } K_1(p^2, \frac{2p^2+1}{3})\, (\mbox{disc $= p^4$}) & \mbox{ for $\gH_{12}$},\\
K_3(1,p^2) \, (\mbox{disc $= p^2$}) \mbox{ or } K_3(p^2,1) \, (\mbox{disc $= p^4$}) & \mbox{ for $\gH_{16}$}.
\end{cases}$$
The labels of $K_1(1, \frac{p^2+2}{3})$ and $K_1(p^2, \frac{2p^2+1}{3})$ are $\lb 12; 2, 2, 2 \rb$ and ~$\lb 12; 2p^2, 2p^2, 2p^2 \rb$ respectively, whereas the labels of $K_3(1,p^2)$ and $K_3(p^2,1)$ are $\lb 16; 1, 1,2,2, p^2\rb$ and $\lb 16; 1, p^2, p^2, 2p^2, 2p^2 \rb$ respectively.  Thus, the labels of $\gH_{12}$ and $\gH_{24}$ are determined.  The symmetries of $K_1(1, \frac{p^2+2}{3})$ and $K_1(p^2, \frac{2p^2+1}{3})$ are all Type II.  However,  for either $K_3(1,p^2)$ or $K_3(p^2,1)$, it has three Type I symmetries and two Type II symmetries.

From Remark \ref{remarkI}, we see that $h_8$ is 0 or 1.  If $h_8 = 1$, the lattices in $\gH_8$ are isometric to
$$\begin{pmatrix}
1 & 0 & 0\\
0 & 2 & 1\\
0 & 1 & \frac{p^2+1}{2}
\end{pmatrix} \, (\mbox{disc $= p^2$}) \mbox { or }
\begin{pmatrix}
2 & 1 & 0\\
1 & \frac{p^2+1}{2} & 0\\
0 & 0 & p^2
\end{pmatrix}\, (\mbox{disc $=p^4$}).$$
The label of these two lattices are $\lb 8; 1, 2, 2p^2 \rb$ and $\lb 8; 2, p^2, 2p^2 \rb$ respectively.   As a result, the label of $\gH_8$ is determined.  The orthogonal system of either lattice consists of one Type I symmetry and two Type II symmetries.  Exactly one of the Type II symmetries is special to the lattice.

Let $\sigma$ be a Type I symmetry.  Then $h_{12}(\sigma) = 0$, and $h_{16}(\sigma) = h_{16}$.  If $h_{16} = 1$, then there are three lattices in $\gH_{16}(\sigma)$, permuted transitively by an order 3 isometry of $N$, and $\sigma$ is special to exactly one of these three lattices.  Suppose that there exists $M$ in $\gH_{8}(\sigma)$.  If $\phi \in O(N)$ and $\phi\sigma\phi^{-1} = \sigma$, then clearly $\phi(M) \in \gH_8(\sigma)$.  Conversely, if $\phi(M) \in \gH_8(\sigma)$, then $\phi\sigma\phi^{-1}$ is the unique Type I symmetry in $O(M)$ and hence $\phi\sigma\phi^{-1} = \sigma$.  This shows that $G_M(\sigma)$ is in fact equal to $C(\sigma)$.  A straightforward computation shows that $C(\sigma)$ has order 16.  Therefore, $\vert \gH_8(\sigma) \vert = 2h_8$ and $\vert \gH_4(\sigma) \vert = 4h_4(\sigma)$.  Consequently,
$$\vert \Gamma_\sigma \vert = 4h_4(\sigma) + 2h_8 + 3h_{16},$$
and so $h_4(\sigma)$ is determined.  Since $\sigma$ is not special to any lattice in $\gH_4(\sigma)$, the label of $\gH_4(\sigma)$ is determined.

Now, let $\tau$ be a Type II symmetry.   Note that $O(N)$ acts, by conjugation, transitively on the set of Type II symmetries.  So, $h_{12}(\tau) = h_{12}$ and $h_8(\tau) = h_8$.    Now, if $h_{16} = 1$, $\tau$ must be a symmetry of at least one of the three lattices in $\gH_{16}$.  But, since each lattice in $\gH_{16}$ has exactly two Type II symmetries, $\vert \gH_{16}(\tau) \vert$ must be 1, whence $\vert \gH_{16}(\tau) \vert = h_{16}$.

For the rest of the discussion, we may assume that $\tau = \tau_{x_1 + x_2}$.  A direct calculation shows that the centralizer of $\tau$ has order 8.
Suppose that $M \in \gH_{12}(\tau)$.  The other two symmetries in $O(M)$ are either $\{\tau_{x_2-x_3}, \tau_{x_1+x_3}\}$ or $\{\tau_{x_2+x_3}, \tau_{x_1-x_3}\}$.  Let $\phi$ be the isometry of $N$ which fixes $x_3$ and switches $x_1$ and $x_3$.  Then $\phi(M) \neq M$ but $\tau \in O(\phi(M))$.  Therefore, $\vert \gH_{12}(\tau) \vert = 2h_{12}$.

If $M \in \gH_8(\tau)$, the orthogonal system of $M$ must be $\{\tau_{x_3}, \tau_{x_1 + x_2}, \tau_{x_1-x_2}\}$.  So, $\tau \in \phi(M)$ if and only of $\phi$ is in the centralizer of $\tau_{x_3}$ which has order 16.  This shows that $\vert \gH_8(\tau)\vert = 2h_8$.  Consequently,
$$\vert \Gamma_\tau \vert = 2h_4(\tau) + 2h_8 + 2h_{12} + h_{16}.$$
Therefore, $h_4(\tau)$ is determined.  Since $\tau$ is not special to any lattice in $\gH_4$, the label of $\gH_4(\tau)$ is determined.
\end{proof}

\subsection{An example}  We illustrate the discussion thus far by computing  the labels of all the classes in the genus of the lattice $K(n): = K_4(1, 6\cdot 7^{2n} + 1)$, $n \geq 0$.  This of course will lead to a class number formula for $K(n)$.

For $n \geq 0$,
$$K(n) \cong \begin{pmatrix}
2 & 0 & -7^n \\
0 & 2 & -7^n\\
-7^n & -7^n & 7^{2n+1}
\end{pmatrix} \cong
\begin{pmatrix}
2 & 0 & -1\\
0 & 2 & -1\\
-1 & -1 & 6\cdot 7^{2n} + 1
\end{pmatrix}.$$
One can easily check that $dK(n) = 24\cdot 7^{2n}$ and $N(n-1): = \Lambda_7(K(n))$ is the lattice $K(n-1)^{7^2}$.  So, $\lambda_7(K(n)) = K(n-1)$.   The class number of $K(0) = K_4(1,7)$ is 1.  The label of $K(n)$ is $\lb 16; 2, 2, 4, 4, 24\cdot 7^{2n} \rb$, and the label of $N(n-1)$ is $\lb 16; 2\cdot 7^2, 2\cdot 7^2, 4\cdot 7^2, 4\cdot 7^2, 24\cdot 7^{2n} \rb$.

For $n \geq 1$,  let $\mathcal G_{2i}(n)$ be the set of lattices in $\gen(K(n))$ whose isometry groups have order $2i$, and $g_{2i}(n)$ be the number of classes in $\mathcal G_{2i}(n)$.  From Example \ref{example}, we see that
$$g_2(1) = 1, \quad g_4(1) = 3, \quad g_8(1) = 0, \quad g_{16}(1) = 1.$$
It is clear that $K(1)$ represents the only class in $\mathcal G_{16}(1)$.  Using row {\bf (1)} of Table II and equations (\ref{order1641}) and (\ref{order1642}), we can show that $h_4(\sigma) = 1$ for each $\sigma \in S(N(0))$.   Thus the labels of the three classes in $\mathcal G_4(1)$ are
$$\lb 4; 2 \rb, \quad \lb 4; 4 \rb, \quad \lb 4; 24 \rb.$$

\begin{lem} \label{example168}
For $n \geq 1$,
\begin{enumerate}
\item[(a)] $g_{16}(n) = 1$, and $K(n)$ represents the only class in $\mathcal G_{16}(n)$;

\item[(b)] $g_8(n) = 0$.
\end{enumerate}
\end{lem}
\begin{proof}
We will provide a proof for part (a); part (b) can be proved in a similar manner.

Part (a) for $n = 1$ is already explained.  For $n \geq 2$, let $M \in \mathcal G_{16}(n)$.  Then $\Lambda_7(M) \cong N(n)$.  But for $n \geq 2$, the $w$ for $N(n)$ is $7^2 = 49$ from Table I.  Therefore, by Table III,
$$g_{16}(n) = h_{16}(N(n-1)) = [49]_2 = 1.$$
Clearly, $K(n)$ represents the only class in $\mathcal G_{16}(n)$.
\end{proof}

For $\ell = 4$ or $16$, let
$$\mathcal G_4^\ell(n+1) = \{ M \in \mathcal G_4(n+1): \lambda_7(M) \in \mathcal G_\ell(n) \}.$$

\begin{lem} \label{7times}
For $n \geq 1$,
\begin{enumerate}
\item[(a)] the label of $\mathcal G_4^4(n+1)$ is the multi-set containing the label of each class in $\mathcal G_4(n)$ repeated $7$ times;

\item[(b)] the label of $\mathcal G_4^{16}(n+1)$ is the multi-set containing $3$ copies of $\lb 4; 2\rb$ and $3$ copies of $\lb 4; 4\rb$.
\end{enumerate}
\end{lem}
\begin{proof}
(a) First of all, an induction argument shows that the label of each class in $\mathcal G_4(n)$ is one of the following: $\lb 4; 2\rb$, $\lb 4; 4\rb$, and $\lb 4; 24 \rb$.  So, if $M \in \mathcal G_4^4(n+1)$, then the label of $\Lambda_7(M)$ is of the form $\lb 4; d \rb$, where $\ord_7(d) = 2$.  From row {\bf (2)} in Table II, we see that each of these lattices will produce 7 classes in $\mathcal G_4(n+1)$ with the same label.

\noindent (b) All the classes in $\mathcal G_4^{16}(n+1)$ descend via $\Lambda_7$ to the class containing $N(n)$.  If $\tau$ is the symmetry in the center of $O(N(n))$, then $\vert \Gamma_\tau \vert = 1$ from row {\bf (2)} of  Table II, and so $h_4(\tau) = 0$ by equation (\ref{order1642}).  If $\sigma$ is any other symmetry in $O(N(n))$, then using row {\bf (2)} of Table II and equation (\ref{order1641}) we can see that $h_4(\sigma) = 3$.
\end{proof}

As a corollary, we obtain the following recursive formula for $g_4(n)$:
$$g_4(n+1) = 7\cdot g_4(n) + 6, \quad n \geq 1,$$
with the initial condition $g_4(1) = 3$.  Therefore,
$$g_4(n) = 4\cdot 7^{n-1} - 1, \quad  n \geq 1.$$

As for $g_2(n+1)$, note that all the lattices  in $\mathcal G_2(n+1)$ descend via $\lambda_7$ to $\mathcal G_2(n)$, $\mathcal G_4(n)$, and $\mathcal G_{16}(n)$.  Using Table III (after calculating the $w$, $f$, and $s$ in each case), we can calculate the contribution from each of these sets: $49 g_2(n)$ from $\mathcal G_2(n)$, $21 g_4(n)$ from $\mathcal G_4(n)$, and 3 from $\mathcal G_{16}(n)$.  Therefore,
$$g_2(n+1) = 49 g_2(n) + 21 g_4(n) + 3, \quad n \geq 1,$$
which implies
$$g_2(n) = 3\cdot 7^{2(n-1)} - 2\cdot 7^{n-1} - \frac{3}{8}(7^{2(n-1)} - 1), \quad n \geq 1.$$
So, finally, the class number of $K(n)$ is
$$g_2(n) + g_4(n) + g_{16}(n) = 3\cdot 7^{2(n-1)} + 2\cdot 7^{n-1} - \frac{3}{8}(7^{2(n-1)} - 1).$$

\section{Labels of stable lattices} \label{stable}

Let $L$ be a primitive ternary lattice and let $q$ be a prime. If $\ord_q(dL)\ge 2$, then one may easily show that $\ord_p(d(\lambda_{eq}(L)))<\ord_q(dL)$, where $e=2$ if $q=2$ and $L$ is even, and $1$ otherwise.  Furthermore if $L$ is odd and $\ord_2(dL)=1$, then $d(\lambda_2(L))$ is odd.  Hence $L$ can be transformed, via a sequence of Watson transformations at primes or at 4, to a primitive ternary lattice $K$ such that $\ord_q(dK) \leq 1$ for all primes $q$, and that $\ord_2(dK) = 1$ if and only if $K$ is even.  For details, see \cite[Corollary 2.8]{ce} and \cite[Lemmas 2.1, 2.3, 2.4]{co}.\footnote{There is a mistake in \cite[Lemma 2.3(2)]{co}.  When $\frac{dM_2}{4} \equiv 5$ mod 8, $\lambda_4(L)_2$ should be $M_2^3\perp N_2^{\frac{1}{2}}$.  This affects neither the results in \cite{co} nor the conclusion we draw here.} For the sake of convenience, we call such a lattice $K$ {\em stable}.    It is clear that a stable lattice is maximal but not vice versa.  Moreover,  if $M$ is another stable lattice such that $dM = dK$, then for any prime $q$, $M_q \cong K_q$ if and only if their Hasse symbols are the same.

Henceforth, $K$ is always a stable ternary lattice.  In this section, we will show that the labels of $\gen(K)$ can be effectively determined.   Let $\h =\begin{pmatrix} 0&1\\1&0 \end{pmatrix}$ and $\plane=\begin{pmatrix} 2&1\\1&2 \end{pmatrix}$.  Then
$$K_2 \cong \begin{cases} \plane \perp \langle 3dK\rangle \qquad &\text{if $K_2$ is anisotropic,}\\
  \h \perp \langle -dK\rangle \qquad &\text{otherwise.}\\
\end{cases}$$
If $K$ is even, then $K_2$ is always isotropic and
$$K_2 \cong \plane \perp \langle 3dK\rangle \cong  \h \perp \langle -dK\rangle.$$

For any integer $a$, let $\nu(a)$ be the number of distinct prime divisors of $a$.  If $q$ is a prime, we define
$$e_q(a): = \begin{cases}
1 & \mbox{ if $q$ divides $a$};\\
0 & \mbox{ otherwise}.
\end{cases}$$
Let $\mathfrak P$ be the product of odd prime divisors $q$ of $dK$ such that $K_q$ is anisotropic, and $\mathfrak Q$ be the product of odd prime divisors $q$ of $dK$ such that $K_q$ is isotropic.  Note that $dK=2^{e_2(dK)}\mathfrak P\mathfrak Q$.  For any positive integer $t$, let $b_t(K)$ be the number of classes in $\gen(K)$ whose isometry groups are of order $t$.

For positive integers $\alpha,\beta,\gamma$, we define
$$
\Phi_K(\alpha)=\displaystyle\prod_{q\mid \mathfrak P}  \left(1-\left(\frac{-\alpha}q\right)\right)
\prod_{q\mid \mathfrak Q} \left(1+\left(\frac{-\alpha}q\right)\right),
$$
and
$$
\begin{array} {rl}
\Phi_K(\alpha,\beta,\gamma)=& \displaystyle \prod_{q\mid \mathfrak P} \left(1+\left(\frac{-\beta\gamma}q\right)\right) \displaystyle\left(1+\left(\frac{-\gamma\alpha}q\right)\right) \left(1+\left(\frac{-\alpha\beta}q\right)\right)\times \\
&\displaystyle\prod_{q\mid \mathfrak Q} \left(1-\left(\frac{-\beta\gamma}q\right)\right) \left(1-\left(\frac{-\gamma\alpha}q\right)\right) \left(1-\left(\frac{-\alpha\beta}q\right)\right).\\
\end{array}
$$
In below, the Hasse symbol of $K$ at a prime $q$ is denoted by $S_q(K)$.

\begin{lem} \label{24case} Up to isometry, there is at most one lattice in $\gen(K)$ whose isometry group is of order $24$. Furthermore,
$$b_{24}(K)=\begin{cases} 0  &\text{if $K$ is odd and $S_2(K) = -1$}; \\
\displaystyle{\frac{e_3(dK)}{2^{\nu(\mathfrak P\mathfrak Q)-1}}\Phi_K(3)}   &\text{otherwise}.
\end{cases}$$
\end{lem}
\begin{proof}
Suppose that $M$ is a stable lattice with $\vert O(M) \vert =24$.  Since $dM$ is squarefree, we have
$$
M = M(b): \cong    \begin{pmatrix} 2&1\\1&2 \end{pmatrix} \perp \langle b\rangle,
$$
for some positive integer $b$.   Note that different choices of $b$ yield lattices in different genera.  If $M(b)$ is odd, then $M(b)_2$ is anisotropic, which happens if and only if $S_2(M(b)) = 1$.  Therefore, if $S_2(K) = -1$ and $K$ is odd, then $M(b)$ is not in $\gen(K)$ for any positive integer $b$.

Now suppose that either $K$ is even or $S_2(K) = 1$.  It is clear that $M(b) \not \in \gen(K)$ for any positive integer $b$ if $3\nmid dK$ or $\plane$ does not split $K_q$ for all primes $q$, which is the same as $e_3(dK)\Phi_K(3) = 0$.  So, we further assume that $e_3(dK)\Phi_K(3) \neq 0$.  Let $b_0$ be chosen such that $dK = 3b_0$.  It is straightforward to check that $M(b_0)$ is in $\gen(K)$.
\end{proof}

\begin{rmk}
In the proof of Lemma \ref{24case}, the label of any $M(b)$ in $\gen(K)$ is known since $b = \frac{dK}{3}$.
\end{rmk}

\begin{lem} \label{12case} Up to isometry, there is at most one lattice in $\gen(K)$ whose isometry group is of order $12$, and its label is $\lb 12; 2,2,2\rb$. Furthermore,
$$b_{12}(K)=\begin{cases} 0 &\text{if $K$ is odd and $S_2(K)=-1$}; \\
\displaystyle{\frac{1-e_3(dK)}{2^{\nu(\mathfrak P\mathfrak Q)}}\Phi_K(3)}  &\text{otherwise}.
\end{cases}
$$
\end{lem}

\begin{proof}
From Section \ref{isometry}, a stable lattice whose isometry group is of order $12$ must be of the form
$$
K_1(a,b) =\begin{pmatrix} 2a&-a&-a\\-a&2a&0\\-a&-a&b\end{pmatrix}
$$
for some positive integers $a$ and $b$.  Its discriminant is $a^2(3b - 2a)$, which should be squarefree.  Therefore, $a = 1$ and the discriminant is $3b - 2 \equiv 1$ mod 3.  Note that $K_1(1,b)_2 \cong \plane \perp \langle 3(3b - 2) \rangle$.  This, in particular, shows that if $K_1(1,b)_2$ is odd, then it is anisotropic and its Hasse invariant is 1.  The label of any one of these lattice is $\lb 12; 2,2,2\rb$.

Suppose that $K$ is even.  It is clear that $\gen(K)$ does not contain any $K_1(1,b)$ when $(1 - e_3(dK))\Phi_K(3) = 0$.  Now, suppose that $e_3(dK) \neq 1$ and $\Phi_K(3) \neq 0$.   Then $3\nmid dK$ and
$$\left( \frac{-3}{q} \right) = \begin{cases}
-1 & \mbox{ if $q \mid \mathfrak P$},\\
1 & \mbox{ if $q \mid \mathfrak Q$}.
\end{cases}$$
By the Quadratic Reciprocity, $\left (\frac{dK}{3} \right ) = (-1)^{\vert \mathfrak P \vert + 1}$.  Since $K_2$ is always isotropic, it follows that $\vert P \vert$ must be even, and hence $dK \equiv 1$ mod 3.  Thus, there exists $b_1$ such that $M: = K_1(1, b_1)$ has discriminant $dK$.  It is direct to check that $M$ is in $\gen(K)$.

The proof of the case when $K$ is odd is similar, and we leave it to the readers.
\end{proof}

\begin{lem} \label{16case} Up to isometry, there is at most one lattice in $\gen(K)$ whose isometry group is of order $16$, and its label is $\lb 16; 1,1,2,2,dK \rb$.  Furthermore,
$$b_{16}(K)=\frac{1-e_2(dK)}{2^{\nu(\mathfrak P\mathfrak Q)}}\Phi_K(1).$$
\end{lem}
\begin{proof} Note that any ternary lattice with an isometry group of order 16 is isometric to
$$K_3(a,b)=\langle a,a,b\rangle \quad \text{or} \quad K_4(a,b)= \begin{pmatrix}2a&0&-a\\0&2a&-a\\-a&-a&b\end{pmatrix},$$
for some suitable integers $a,b$.  Note that $d(K_3(a,b))=a^2b$ and $d(K_4(a,b))=4a^2(b-a)$.  Thus, if $M$ is a stable lattice with $\vert O(M)\vert = 16$, then $M \cong K_3(1,b)$ with $b > 1$.  We may now proceed as in the proofs of the previous two lemmas.
\end{proof}

Let $M$ be a ternary lattice with $\vert O(M)\vert =8$.  Then the symmetries in $O(M)$ form an orthogonal system $\{\tau_{z_1}, \tau_{z_2}, \tau_{z_3}\}$, with primitive mutually orthogonal vectors $z_1, z_2, z_3$.  Let $L$ be the sublattice spanned by $z_1, z_2, z_3$.  By Lemma \ref{trivial}, $M_p = L_p$ for all odd primes $p$.  In addition, it is direct to check that $2v \in L$ for every $v \in M$.  Therefore, $M$ is obtained from $L$ by adjoining one or more vectors of the form $\frac{\epsilon_1z_1 + \epsilon_2z_2 + \epsilon_3z_3}{2}$, $\epsilon_i = 0$ or 1 for each $i$.  As a result, $M$ is isometric to one of the following lattices
$$M_1(a,b,c)=\langle a,b,c\rangle, \  \ M_2(a,b,c)=\langle a\rangle \perp \begin{pmatrix} \frac{b+c}2&\frac{b-c}2\\
 \frac{b-c}2&\frac{b+c}2\end{pmatrix}$$
or
$$M_3(a,b,c)=\begin{pmatrix}2a&0&a\\0&2b&b\\a&a&\frac{a+b+c}2\end{pmatrix}, \  \ M_4(a,b,c)=\begin{pmatrix}4a&2a&2a\\2a&a+b&a\\2a&a&a+c\end{pmatrix},$$
for some suitable integers $a, b$, and $c$.  A simple calculation shows that the discriminants of $M_3(a,b,c)$ and $M_4(a,b,c)$ are divisible by 4.  Thus, if $M$ is stable, then $M \cong M_1(a,b,c)$ or $M_2(a,b,c)$ for some suitable positive integers $a, b, c$.

For the convenience of discussion, let $\mathfrak T$ be the set of triples $(a,b, c)$ of positive integers such that $abc = dK$, and we define
\begin{eqnarray*}
\mathfrak R_1 & = &\{(a,b,c)\in \mathfrak T:  a > b > c \},\\
\mathfrak R_2 & = & \{(a,b,c) \in \mathfrak T:  b > c, (b,c) \neq (3,1), a \equiv 2 \mbox{ and } 4, \mbox{ and } bc \equiv 3 \mbox{ and } 4\},\\
\mathfrak R_3 & = & \{(a,b,c) \in \mathfrak T: b > c, \mbox{ and } (b,c) \neq (3,1)\}.
\end{eqnarray*}

\begin{lem} \label{8case} If $K$ is even, then
$$b_8(K)=\displaystyle \sum_{(a,b,c) \in \mathfrak R_2} \frac{1}{2^{\nu(\mathfrak P\mathfrak Q)}}\Phi_K(a,2b,2c),$$
and if $K$ is odd, then
$$b_8(K)=\displaystyle \sum_{(a,b,c) \in \mathfrak R_1} \frac{1}{2^{\nu(\mathfrak P\mathfrak Q)}}\Phi_K(a,b,c) + \displaystyle \sum_{(a,b,c) \in \mathfrak R_3} \frac{1}{2^{\nu(\mathfrak P\mathfrak Q)}}\Phi_K(a,2b,2c).$$
\end{lem}
\begin{proof}
Suppose that there is a lattice $M$ in $\gen(K)$ such that $\vert O(M)\vert =8$.  We first assume that $K$ is even.  Then $M \cong M_2(a,b,c)$  for a unique triple $(a,b,c)\in \mathfrak R_2$.  Since $M_q \cong \langle a,2b,2c\rangle$ for any prime $q$ dividing $\mathfrak{PQ}$, therefore
$$2^{\nu(\mathfrak{PQ})}= \Phi_K(a,2b,2c).$$

Conversely, suppose that $(a,b,c) \in \mathfrak R_2$ and $\Phi_K(a,b,c) = 2^{\nu(\mathfrak{PQ})}$, and let $M$ be $M_2(a,b,c)$.  Then $K_q \cong \langle a,2b,2c\rangle \cong M_q$ for every prime $q$ dividing $\mathfrak{PQ}$.  By the Hilbert Reciprocity, $K_2$ is also isometric to $M_2$, hence $M \in \gen(K)$.

The proof of the case when $K$ is odd is similar, and we leave it to the readers.
\end{proof}

\begin{rmk}
The label of any $M \in \gen(K)$ with $\vert O(M) \vert = 8$ can be determined.  For, if $M \cong M_1(a,b,c)$ with $(a,b,c) \in \mathfrak R_1$, then the label of $M$ is $\lb 8; c, b, a \rb$.  On the other hand, $M \cong M_2(a,b,c)$ with $(a,b,c) \in \mathfrak R_2$ or $\mathfrak R_3$, then $Q_M(\sigma) = a, 2b$, or $2c$ for any $\sigma \in S(M)$, and hence the label of $M$ can also be determined.
\end{rmk}


Let $\tau_x$ be a symmetry of $K$, where $x$ is a primitive vector in $K$.  If $Q(x)$ is odd, then $\z x$ splits $K$ by Lemma \ref{trivial}. If $Q(x)=2m$ for some integer $m$, then either $\z x$ splits $K$ or there is a basis $x=x_1, x_2,x_3$ of $K$ such that
\begin{equation} \label{matrixm}
(B(x_i,x_j))=\begin{pmatrix} 2m&m&0\\m&Q(x_2)&B(x_2,x_3)\\0&B(x_2,x_3)&Q(x_3)\end{pmatrix}.
\end{equation}
Therefore $m$ divides $\mathfrak{PQ}$.

From now on, $\delta$ is either $1$ or $2$.  For any integer $t$ and lattice $L$, $r(t, L)$ denotes the number of representations of $t$ by $L$.

\begin{lem} \label{keylem} Let $m$ be a positive odd squarefree integer.  There exists $\tau_x \in O(K)$ such that $Q(x)=\delta m$ if and only if $\delta$ is represented by $\lambda_m(K)$. Furthermore
$$
\big \vert \{ \tau_x \in O(K) : x \in K, \ Q(x)=\delta m\}\big\vert = \frac{1}{2}r(\delta,\lambda_m(K)).
$$
\end{lem}
\begin{proof} Suppose that $\tau_x \in O(K)$ and $Q(x)=\delta m$. If $\z x$ splits $K$, then $K =\langle \delta m\rangle \perp \widetilde{K}$ for some binary sublattice $\widetilde{K}$ of $K$. Since $\gcd(m, d\widetilde{K})=1$, $\Lambda_m(K)=\langle \delta m\rangle \perp m\widetilde{K}$ and hence
$\lambda_m(K)=\langle \delta \rangle \perp \widetilde{K}^m.$  On the other hand, if $\z x$ does not split $K$, then clearly $\delta=2$, and there is a basis $x_1, x_2, x_3$ of $K$ satisfying (\ref{matrixm}).  In this case,
$$\lambda_m(K) \cong \begin{pmatrix} 2&m&0\\m&mQ(x_2)&mB(x_2,x_3)\\0&mB(x_2,x_3)&mQ(x_3)\end{pmatrix},$$
which clearly represents 2.

Conversely, suppose that $\delta$ is represented by $\lambda_m(K)$.  We assume that $\langle \delta \rangle$ does not split $\lambda_m(K)$; the other case can be done similarly. Thus, $\delta=2$ and there is a basis $x_1, x_2, x_3$ of $\lambda_m(K)$ such that
$$
(B(x_i,x_j))=\begin{pmatrix} 2&1&0\\1&a&b\\0&b&c\end{pmatrix},
$$
for some integers $a,b,c$. If $q$ is a prime dividing $m$, then $\ord_q(dK) = 1$ and hence $\lambda_m(K)_q \cong \langle \epsilon_1,q\epsilon_2,q\epsilon_3\rangle$, where $\epsilon_i \in \z_q^{\times}$ for every $i$, by Lemma \ref{lambda}.  So, whenever $q$ is a prime divisor of $m$, $\z_q mx_1$ would be an orthogonal summand of $\Lambda_m(\lambda_m(K))_q$.  Consequently,
$$\Lambda_m(\lambda_m(K))_q = \begin{cases}
\lambda_m(K)_q & \mbox{ if $q\nmid m$},\\
\z_q mx_1\perp (\z_q (x_1 - 2x_2) + \z_q x_3) & \mbox{ if $q \mid m$}.
\end{cases}$$
This implies
\begin{eqnarray*}
\Lambda_m(\lambda_m(K)) & = & \z (mx_1)+\z\left(\frac{m-1}{2}x_1+x_2\right)+\z x_3 \\
    & \cong &  \begin{pmatrix} 2m^2 & m^2 & 0\\ m^2 & \frac{m^2-1}{2} + a & b\\ 0 & b & c \end{pmatrix}.
\end{eqnarray*}
However, since $m$ is squarefree and odd, $K \cong \lambda_m^2(K)$ by Lemma \ref{lambda}, and the latter is  $\Lambda_m(\lambda_m(K))^{\frac{1}{m}}$. It is easy to see that $\tau_{mx_1}$ is a symmetry in $O(K)$ with $Q(mx_1) = 2m$ (note that the quadratic form on $K$ is the one on $\lambda_m(K)$ scaled by $\frac{1}{m}$).
\end{proof}

\begin{lem} \label{compw}
The mass of $K$, $\mathfrak w(K)$, is equal to
$$\frac{\epsilon}{2^{\nu(\mathfrak{PQ})}}\displaystyle \prod_{p\mid \mathfrak P} (p-1)\prod_{p\mid \mathfrak Q} (p+1),$$
where $\epsilon=\frac{1}{16}, \frac{1}{48}$ or $\frac{1}{24}$ if $K_2$ is odd isotropic, odd anisotropic, or even, respectively.
\end{lem}
\begin{proof}
Then lemma follows directly from the Minkowski-Siegel mass formula. For the computation of local densities, see \cite[Theorem 5.6.3]{ki}.
\end{proof}

 Let $h_E$ be the class number of the quadratic field $E=\q(\sqrt{-\delta d(\lambda_m(K))})$, and $\mu_E$ be the number of roots of unity in $E$.

\begin{lem} \label{locald}
Suppose that $\delta$ is represented by $\gen(\lambda_m(K))$. Then
$$
\displaystyle \sum_{[K_i] \in \gen(\lambda_m(K))} \frac{r(\delta,K_i)}{\vert O(K_i)\vert} = t_{m,\delta}\cdot 2^{\nu(m)-\nu(\mathfrak{PQ})}\cdot \frac{h_E}{\mu_E},
$$
where  $t_{m,\delta}$ is given in Table V.
\begin{table} [ht]
\begin{centerline}{\tabcolsep=3pt
\begin{tabular}{|c|c|c||c|c|c|}
\hline \rule[-2mm]{0mm}{8mm}
$\lambda_m(K)_2$&$ \ \delta \ $&$t_{m,\delta}$&$\lambda_m(K)_2$&$ \ \delta \ $&$t_{m,\delta}$\\
\hline \hline
\rule[-2mm]{0mm}{6mm}  $\langle1,1,3\rangle$&$1$&$3$&$\langle3,3,3\rangle$&$1$&$1$ \\
\hline
\rule[-2mm]{0mm}{6mm}  $\langle1,1,7\rangle$&$1$&$2$&  $dK\equiv 1 \pmod 4$&$1$&$\frac{1}{2}$\\
\hline
\rule[-2mm]{0mm}{6mm}  $\plane\perp \langle2\rangle$&$2$&$4$&$\plane\perp \langle14\rangle$&$2$&$1$ \\
\hline
\rule[-2mm]{0mm}{6mm}  $\plane\perp \langle6\rangle$&$2$&$1$&$\plane\perp \langle10\rangle$&$2$&$2$ \\
\hline
\rule[-2mm]{0mm}{6mm}{\rm odd} &$2$&$\frac{1}{2}$&{}&{}&{} \\
\hline
\end{tabular}}
\end{centerline}
\vskip 0.2cm
\center{\rm Table V}
\end{table}
\end{lem}

\begin{proof} This follows from a direct computation, using the Minkowski-Siegel mass formula for representations of integers by ternary quadratic forms and
$$
\frac{h_E}{\mu_E}=(2\pi)^{-1}\vert d \vert^{\frac{1}{2}} L\left(1,\displaystyle\left(\frac{d_E}{.}\right)\right),
$$
where $d_E$ is the discriminant of $E$.  Note that $\mathfrak w(K)=\mathfrak w(\lambda_m(K))$ and $h(K)=h(\lambda_m(K))$ for any integer $m$ dividing $\mathfrak{PQ}$. For the computation of local densities, see \cite{y}.
\end{proof}

We define
$$
b_{k,\delta m}=\displaystyle \sum_{\stackrel{[\widetilde{K}] \in \gen(K)}{\vert O(\widetilde{K})\vert=k}} \big \vert \{ \tau_x \in O(\widetilde{K}) :  Q(x)=\delta m\}\big \vert.
$$
Suppose that $\delta$ is represented by $\gen(\lambda_m(K))$. Since
$$
\frac{2}{4}b_{4,\delta m}+\frac{2}{8}b_{8,\delta m}+\frac{2}{12}b_{12,\delta m}+\frac{2}{16}b_{16,\delta m}+\frac{2}{24}b_{24,\delta m}=t_{m,\delta}\cdot 2^{\nu(m)-\nu(\mathfrak{PQ})}\cdot \frac{h_E}{\mu_E},
$$
by Lemma \ref{keylem}, we may effectively compute the number of classes of lattices in $\gen(K)$ with label $\lb 4;\delta m\rb$ once we know the labels of all the classes of lattices whose isometry groups are of order greater than $4$.

At last, in order to determine the number of classes in $\gen(K)$ with label $\lb 2 \rb$, all we need now is the class number of $K$ which is given by the following lemma.

\begin{thm} \label{maint} The class number $h(K)$ of $K$ is equal to
$$
2\mathfrak w(K)+ \displaystyle\sum_{\stackrel{m\mid \mathfrak{PQ}, \ \delta \in \{1,2\}}{\delta \ra \gen(\lambda_m(K))}}t_{m,\delta}\cdot 2^{\nu(m)-\nu(\mathfrak{PQ})}\cdot \frac{h_E}{\mu_E}+\frac{1}{3}\left(b_{12}(K)+b_{24}(K)\right)+\frac{1}{4}b_{16}(K).
$$
\end{thm}

\begin{proof} Note that
$$
\sum_{m\mid \mathfrak{PQ}, \ \delta \in \{1,2\}}b_{k,\delta m}=\mathfrak s_k\cdot b_k(K),
$$
where $\mathfrak s_k$ is the number of symmetries of a lattice in $\gen(K)$ whose isometry group has order $k$.  The values of $\mathfrak s_k$  is  $1,3,3,5$ or $7$ according to $k=4,8,12,16$ or $24$, respectively. Hence
$$
\frac12b_4+\frac34b_8+\frac12b_{12}+\frac58b_{16}+\frac7{12}b_{24}= \displaystyle\sum_{\stackrel{m\mid \mathfrak{PQ}, \ \delta \in \{1,2\}}{\delta \ra \gen(\lambda_m(K))}}t_{m,\delta}\cdot 2^{\nu(m)-\nu(\mathfrak{PQ})}\cdot \frac{h_E}{\mu_E}.
$$
Here $b_k=b_k(K)$ for each $k$. Therefore
\begin{eqnarray*}
h(K)&=&b_2+b_4+b_8+b_{12}+b_{16}+b_{24}\\
&=&2\mathfrak w(K)+\frac12b_4+\frac34b_8+\frac56b_{12}+\frac78b_{16}+\frac{11}{12}b_{24}\\
&=&2\mathfrak w(K)+ \displaystyle\sum_{\stackrel{m\mid \mathfrak{PQ}, \ \delta \in \{1,2\}}{\delta \ra \gen(\lambda_m(K))}}\eta_{m,\delta}\cdot 2^{\nu(m)-\nu(\mathfrak{PQ})}\cdot \frac{h_E}{\mu_E}\\
 & & +\frac13\left(b_{12}(K)+b_{24}(K)\right)+\frac14b_{16}(K).
\end{eqnarray*}
This completes the proof.
\end{proof}

\appendix

\section{$p = 3$ and $\vert O(N)\vert = 24$}

In this appendix, we treat the case $p = 3$ and $\vert O(N) \vert = 24$.  So, $\Lambda_3(L) = N \cong K_2(a,b)$, and we let $L_3 \cong \langle \epsilon_1, 3^\alpha \epsilon_2, 3^\beta \epsilon_3 \rangle$ for some $\epsilon_1, \epsilon_2, \epsilon_3 \in \z_3^\times$.   Recall that $\ord_3(dL) = \alpha + \beta \geq 2$ is always assumed, and that for $1\leq i < j \leq 3$ the integer $e_{ij}$ is defined to be $1$ or $-1$, depending on whether $-\epsilon_i\epsilon_j$ is a square or not.

Now, $dN = 3a^2b$, and $N_3 \cong \langle 2a, 6a, b \rangle$ which is isometric to $\Lambda_3(L)_3$.  This implies that both $a$ and $b$ are divisible by 3, and that Cases {\bf (1)} and {\bf (6)} in Table I cannot occur.

For simplicity, we denote each $h_{i}(N)$ by $h_i$.  Inside $O^+(N) \cong \z_2\oplus D_3$, there are two isometries of order 3 and two isometries of order 6.  If $M \in \Gamma_3^L(N)$ and $\vert O(M) \vert = 12$, then $O^+(M) \cong D_3$  does not contain any isometry of order 6.  Hence, by (\ref{h1}) and (\ref{first}),
\begin{eqnarray}
h_2 + h_4 + h_8 + h_{12} + h_{24} & = & \frac{1}{12}(w + f + 4h_{12} + 4h_{24}) \label{24eqn1}\\
12h_2 + 6h_4 + 3h_8 + 2h_{12} + h_{24} & = & w.\label{24eqn2}
\end{eqnarray}
There are seven symmetries $\sigma$ in $O(N)$: three of them with $Q_N(\sigma) = 2a$, three of them with $Q_N(\sigma) = 6a$, and one of them, which is in the center of $O(N)$,  with $Q_N(\sigma) = b$.  This, in particular, means that the number $f$ can be determined by using Table II.  The value of $w$ is determined in Table I.

\begin{lem} \label{app24}
Suppose that $M \in \Gamma_3^L(N)$.  If $\vert O(M) \vert = 24$, then $M = K_2(a, \frac{b}{9})$ or $K_2(\frac{a}{3}, b)$.  Moreover,
\begin{enumerate}
\item[(a)] if $M = K_2(\frac{a}{3}, b)$, then $\ord_3(a) = 1$ and $\alpha + 1 = \beta$;

\item[(b)] if $M = K_2(a, \frac{b}{9})$, then $\ord_3(b) = 2$ and $\alpha = 1$;

\item[(c)] $K_2(\frac{a}{3}, b)$ and $K_2(a, \frac{b}{9})$ are in the same genus if and only if $\ord_3(a) = 1$, $\ord_3(b) = 2$, and $\frac{a}{3} \not \equiv \frac{b}{9}$ mod $3$, which happens only in Case {\bf (4)} of Table I.
\end{enumerate}

\end{lem}
\begin{proof}
Since $\vert O(M) \vert = 24$, there are relatively prime positive integers $c$ and $d$ such that $M = K_2(c,d)$.  Since $dM = 3c^2d$ and $\ord_3(dM) = \ord_3(dL) \geq 2$, either $3\mid c$ or $3\mid d$.  In the first case,
$$N = \Lambda_3(M) = K_2(c, 9d).$$
So, $a = c$ and $b = 9d$; hence $M = K_2(a, \frac{b}{9})$.  On the other hand, if $3 \mid d$, then $\Lambda_3(M) = K_2(3c, d)$, which means that $a = 3c$ and $d = b$.  Thus, $M = K_2(\frac{a}{3}, b)$ in this case.

Parts (a), (b), and (c) are direct consequences of an examination of the local structure of the lattices at the prime 3.
\end{proof}

\begin{lem} \label{app12}
Suppose that $M \in \Gamma_3^L(N)$.  If $\vert O(M) \vert = 12$, then $M = K_1(a, \frac{6a + b}{9})$, $\ord_3(6a + b) = 2$, and $\alpha + 1 = \beta$.
\end{lem}
\begin{proof}
From Section \ref{isometry}, it follows that $M = K_1(c,d)$ for some relatively prime integers $c$ and $d$.  Since $dM = c^2(3d - 2c)$ and $\ord_3(dM) \geq 2$, $c$ is divisible by 3 but $d$ is not. Consequently,
$$\Lambda_3(M) \cong \begin{pmatrix} 2c & -c & -3c\\ -c & 2c & 0 \\ -3c & 0 & 9d \end{pmatrix} \cong K_2(c, 9d - 6c),$$
hence $a = c$ and $b = 9d - 6c$.
\end{proof}

It follows from Table I (and the fact that $p = 3$ here) that $w \leq 9$.  So, we can deduce from (\ref{24eqn2}) that $h_2$ is always zero.
By Lemmas \ref{app24} and \ref{app12}, $h_{12} \leq 1$ and $h_{24} \leq 2$; furthermore, $h_{24} \leq 1$ unless we are in Case {\bf (4)}.

Suppose that we are not in Case {\bf (3)} or in Case {\bf (4)}.   Then, from Table I, $w$ is divisible by 3.  Therefore, $2h_{12} + h_{24} \equiv 0$ mod 3 by (\ref{24eqn2}), and hence $(h_{12}, h_{24}) = (0,0)$ or $(1, 1)$.    It is now ready to determine the remaining $h_i$ for all the cases in Table I.  We remind the readers that Case {\bf (1)} and {\bf (6)} cannot occur.
\vskip 2mm

\noindent Case {\bf (2)}\quad  In this case, $\alpha = 0$ and $\beta \geq 3$.  Therefore, by Lemmas \ref{app24} and \ref{app12}, both $h_{12}$ and $h_{24}$ are 0. Now, since $\Lambda_3(L)_3 = N_3 = \langle 2a, 6a, b \rangle$, we must have $\beta = 3$ and $\ord_3(a) = \ord_3(b) = 2$.  So, by Table II, $f = 15$.  Also, $w = 9$ from Table I.  Thus, by (\ref{24eqn1}) and (\ref{24eqn2}), both $h_4$ and $h_8$ are equal to 1.
\vskip 2mm

\noindent Case {\bf (3)} \quad  We know from Table I that $w = 1$.  Therefore, by (\ref{24eqn2}), $h_4 = h_8 = h_{12} = 0$ and $h_{24} = 1$.  The only lattice in $\Gamma_3^L(N)$ is either $K_2(\frac{a}{3}, b)$ or $K_2(a, \frac{b}{9})$, and we choose the one which is in $\gen(L)$.
\vskip 2mm

\noindent Case {\bf (4)} \quad  If $e_{13} = 1$, then $w = 1$ from Table I,  which implies that $h_{24} = 1$ and $h_4 = h_8 = h_{12} = 0$.  As in Case {\bf (3)}, the only lattice in $\Gamma_3^L(N)$ can be determined.

On the other hand, if $e_{13} = -1$, then $w = 2$ from Table I.  Moreover, $\frac{a}{3} \not \equiv \frac{b}{9}$ mod 3.  Thus, if $\frac{b}{9}$ and $\epsilon_1$ are in the same square class in $\q_3$, then $h_{24} = 2$ and $h_4 = h_8 = h_{12} = 0$ by (\ref{24eqn2}).  Otherwise, $h_{12} = 1$ and $h_4 = h_8 = h_{24} = 0$.
\vskip 2mm

\noindent Case {\bf (5)} \quad  In this case $\alpha = 1$ and $\beta \geq 3$.  By Table I, $w = 3$.  Since $\Lambda_3(L)_3 = N_3$, therefore  $\ord_3(a) = 1$ and $\ord_3(b) \geq 3$.  So, by Table II, $f = 13$.   Using (\ref{24eqn1}) and (\ref{24eqn2}), we can deduce that $h_{12} = h_{24} = 1$ and $h_4 = h_8 = 0$.  The lattice with isometry group of order 24 is $K_2(\frac{a}{3}, b)$.
\vskip 2mm

\noindent Case {\bf (7)} \quad In this case, $\alpha = 2$ and $\beta \geq 3$.  Because of $\Lambda_3(L)_3 = N_3$, we can deduce that $\beta = 3$, $\ord_3(a) = \ord_3(b)  =2$, and $w = 3$ or $6$ depending on whether $e_{12} = 1$ or $-1$.

If $w = 3$, then $h_4 = 0$ and
$$(h_8, h_{12}, h_{24}) = \begin{cases}
(0,1,1) & \mbox{ if $K_2(a, \frac{b}{9}) \in \gen(L)$};\\
(1,0,0) & \mbox{ otherwise}.
\end{cases}$$

If $w = 6$, then $L_3$ must be isometric to $\langle \frac{a}{3}, 2b, 6a \rangle$ and $ab$ is not a square in $\q_3$.  Thus, $f = 6$ by Table II, and
$$(h_4,h_8,h_{12}, h_{24}) = \begin{cases}
(0,1,1,1) & \mbox{ if  $K_2(a, \frac{b}{9}) \in \gen(L)$}\\
(1,0,0,0) & \mbox{ otherwise},
\end{cases}$$
by (\ref{24eqn1}) and (\ref{24eqn2}).

In any case, if $h_{24} = 1$, then the only lattice in $\Gamma_3^L(N)$ with isometry group of order 24 is $K_2(a, \frac{b}{9})$.
\vskip 2mm

\noindent Case {\bf (8)}\quad Again, using the fact that $\Lambda_3(L)_3 = N_3$, one can show that $\alpha = 3$, $\ord_3(a) = 2$, and $\ord_3(b) \geq 3$.  This implies that $h_{12} = h_{24} = 0$, and that $f = 15$ form Table II.  Since $w = 9$, we can use (\ref{24eqn1}) and (\ref{24eqn2}) to obtain $h_4 = h_8 = 1$.
\vskip 2mm

We now turn our attention to the labels of the classes in $\Gamma_3^L(N)$.  For those lattices in $\gH_{12}$ or $\gH_{24}$, their labels are determined by Lemmas \ref{app24} and \ref{app12}.  There are only three cases, namely Cases {\bf (2)}, {\bf (7)}, and {\bf (8)}, in which $h_4$ and $h_8$ are not zero.  We will determine the labels of these classes in these three cases separately.
\vskip 2mm

\noindent Case {\bf (2)}\quad Since $h_8 = 1$, it is clear that $L_3 \cong \langle \frac{b}{9}, \frac{2a}{9}, 6a \rangle$.  Moreover, for any $M \in \gH_8$, the values of $Q_M(\sigma)$, $\sigma \in S(M)$, are $\frac{b}{9}, \frac{2a}{9}$, and $6a$ respectively.  Using (\ref{24first}), we find that if $\sigma \in S(N)$ with $Q_N(\sigma) = 2a$, then $h_4(\sigma) = 1$.  Since $h_4 = 1$, therefore the label of $\gH_4$ is $\lb 4; \frac{2a}{9} \rb$.
\vskip 2mm

\noindent Case {\bf (7)} \quad Suppose that $w = 3$.  It suffices to deal with the case when $K_2(a, \frac{b}{9}) \not \in \gen(L)$.  Since $h_8 = 1$, $L_3$ must be isometric to $\langle \frac{2a}{9}, b, 6a \rangle$; hence the label of $\gH_8$ is determined as in Case {\bf (2)}.

Suppose that $w = 6$.  If $K_2(a, \frac{b}{9}) \in \gen(L)$, then $L_3 \cong \langle \frac{b}{9}, 2a, 6a \rangle$ and the label of $\gH_8$ is determined.  However, if $K_2(a, \frac{b}{9}) \not \in \gen(L)$, then $L_3 \langle \frac{a}{9}, 2b, 6a \rangle$ with $ab$ not a square in $\q_3$.  By (\ref{24first}), $h_4(\sigma) = 1$ when $Q_N(\sigma) = 6a$.  Thus the label of $\gH_4$ is $\lb 4; 6a \rb$.
\vskip 2mm

\noindent Case {\bf (8)}\quad In this case, $\alpha = 3 \leq \beta$, $\ord_3(a) = 2$, and $\ord_3(b) \geq 3$.  This shows that $L_3 \cong \langle \frac{2a}{9}, 6a, b \rangle$.  Hence the label of $\gH_8$ is determined.  It follows from (\ref{24first}) that $h_4(\sigma) = 1$ if $Q_N(\sigma) = 6a$.  Therefore, the label of $\gH_4$ is $\lb 4; 6a \rb$.

\end{document}